\RequirePackage[l2tabu, orthodox]{nag}

\documentclass[12pt, reqno]{amsart} 
\usepackage{amsthm, amsmath, mathtools}
\usepackage{amssymb}
\usepackage{amscd}
\usepackage{enumerate}
\usepackage{mathrsfs} 
\usepackage[curve,matrix,arrow,frame,tips]{xy}
\usepackage{colonequals}
\usepackage{verbatim}
\usepackage{pdfsync}
\usepackage{graphicx}
\usepackage[usenames,dvipsnames]{color}

\newcommand{\defi}[1]{\textsf{#1}} 

\usepackage{url}

\usepackage{fullpage}

\usepackage[normalem]{ulem}

\numberwithin{equation}{section} 

\theoremstyle{plain}
\newtheorem{theorem}[equation]{Theorem} 
\newtheorem{proposition}[equation]{Proposition}
\newtheorem{lemma}[equation]{Lemma}

\newtheorem{corollary}[equation]{Corollary}

\newtheorem{question}[equation]{Question}

\theoremstyle{definition}
\newtheorem{definition}[equation]{Definition}
\newtheorem{example}[equation]{Example}

\theoremstyle{remark}
\newtheorem{remark}[equation]{Remark}

\newcommand{\kbar}{\overline{k}}

\newcommand{\F}{\mathbb{F}}

\newcommand{\iso}{\simeq}

\newcommand{\XX}{\mathscr{X}}
\newcommand{\OO}{\mathscr{O}}

\newcommand{\Z}{\mathbb{Z}}

\newcommand{\intersect}{\cap}
\newcommand{\isom}{\simeq}
\newcommand{\isomto}{\stackrel{\sim}\to}
\newcommand{\injects}{\hookrightarrow}
\newcommand{\surjects}{\twoheadrightarrow}
\newcommand{\union}{\cup}

\newcommand{\calA}{\mathcal{A}}

\newcommand{\calX}{\mathcal{X}}

\newcommand{\Jbar}{\overline{J}}
\newcommand{\Pbar}{\overline{P}}

\DeclareMathOperator{\Spec}{Spec}

\DeclareMathOperator{\SL}{SL}
\DeclareMathOperator{\GL}{GL}

\DeclareMathOperator{\Gal}{Gal}

\DeclareMathOperator{\Stab}{Stab}

\DeclareMathOperator{\alg}{alg}

\DeclareMathOperator{\Disp}{Disp}

\DeclareMathOperator{\Aut}{Aut}
\DeclareMathOperator{\PGL}{PGL}
\DeclareMathOperator{\PSL}{PSL}
\DeclareMathOperator{\Frac}{Frac}

\newcommand{\Aff}{\mathbb{A}}

\newcommand{\PP}{\mathbb{P}}

\renewcommand{\to}{\longrightarrow}
\newcommand{\mto}{\longmapsto}

\usepackage{microtype} 
\usepackage[pdfauthor={Frauke Bleher, Ted Chinburg, Bjorn Poonen, Peter Symonds}]{hyperref}

\begin{document}

\title{Automorphisms of Harbater--Katz--Gabber curves }

\author[F. Bleher]{Frauke M. Bleher*}\thanks{*Supported by NSA Grant \# H98230-11-1-0131 and NSF Grant \# DMS-1360621.}
\address{Frauke M. Bleher\\ Department of Mathematics\\University of Iowa\\Iowa City, IA 52242, U.S.A.}
\email{frauke-bleher@uiowa.edu}

\author[T. Chinburg]{Ted Chinburg**}\thanks{**Supported by NSF Grants \# DMS-1265290 and DMS-1360767, SaTC grant CNS-15136718 and a grant from the Simons Foundation (338379 to Ted Chinburg).}
\address{Ted Chinburg\\ Department of Mathematics\\University of Pennsylvania\\ Philadelphia, PA 19104, U.S.A.}
\email{ted@math.upenn.edu}

\author[B. Poonen]{Bjorn Poonen\dag}\thanks{\dag Supported by NSF Grants \# DMS-1069236 and DMS-1601946 and a grants from the Simons Foundation (340694 and 402472 to Bjorn Poonen).}
\address{Bjorn Poonen\\Department of Mathematics\\ Massachusetts Institute of Technology\\
Cambridge, MA 02139, U.S.A.}
\email{poonen@math.mit.edu}

\author[P. Symonds]{Peter Symonds} 
\address{Peter Symonds\\ School of Mathematics\\
University of Manchester\\Oxford Road,
Manchester M13 9PL
\\Manchester M13 9PL
United Kingdom}
\email{Peter.Symonds@manchester.ac.uk}


\begin{abstract}
Let $k$ be a perfect field of characteristic $p > 0$, 
and let $G$ be a finite group.
We consider the pointed $G$-curves over $k$ associated 
by Harbater, Katz, and Gabber 
to faithful actions of $G$ on $k[[t]]$ over $k$.
We use such ``HKG $G$-curves''
to classify the automorphisms of $k[[t]]$ of $p$-power order
that can be expressed by particularly explicit formulas,
namely those mapping $t$ to a power series 
lying in a $\Z/p\Z$ Artin--Schreier extension of $k(t)$.
In addition, we give necessary and sufficient criteria 
to decide when an HKG $G$-curve with an action
of a larger finite group $J$ is also an HKG $J$-curve.
 \end{abstract}

\maketitle

\section{Introduction}
\label{s:intro}

Let $k$ be a field, let $k[[t]]$ be the power series ring,
and let $\Aut(k[[t]])$ be its automorphism group as a $k$-algebra.
When the characteristic of $k$ is positive,
$\Aut(k[[t]])$ contains many interesting finite subgroups.
One way to construct such subgroups is 
to start with an algebraic curve $X$ 
on which a finite group $G$ acts with a fixed point $x$ 
having residue field $k$; 
then $G$ acts on the completion $\hat{\mathcal{O}}_{X,x}$ 
of the local ring of $x$ at $X$, and $\hat{\mathcal{O}}_{X,x}$
is isomorphic to $k[[t]]$ for any choice of uniformizing parameter $t$ at $x$.
In fact, results of Harbater~\cite[\S 2]{harbater} 
and of Katz and Gabber~\cite[Main Theorem~1.4.1]{katz} show that 
\emph{every} finite subgroup $G$ of $\Aut(k[[t]])$ arises in this way.
Their results connect the \'etale fundamental group of $\Spec(k((t)))$
to that of $\PP^1_k - \{0,\infty\}$.
See Section~\ref{s:HKGtheorem} for further discussion. 
The value of this technique is that one can study local questions about
elements of $\Aut(k[[t]])$ using global tools such
as the Hurwitz formula for covers of curves over $k$.  

In this paper we use the above method to study two closely related problems
when $k$ is a perfect field of characteristic $p > 0$, 
which we assume for the rest of this paper.
The first problem, 
described in Section~\ref{s:finiteorder}, 
is to find explicit formulas for $p$-power-order elements 
$\sigma$ of $\Aut(k[[t]])$.
In particular, we study $\sigma$ that are ``almost rational''
in the sense of Definition~\ref{D:almost rational}.
Our main result in this direction, Theorem~\ref{T:almost rational},
classifies all such $\sigma$.

The second problem, described in Sections \ref{s:HKGdefs} and~\ref{s:extrauts},
is to study the full automorphism group
of the so-called Harbater--Katz--Gabber $G$-curves (HKG $G$-curves), 
which are certain curves $X$ with a $G$-action as above.
One reason for this study is that it turns out that 
almost rational automorphisms arise from HKG $G$-curves $X$ 
for which $\Aut(X)$ is strictly larger than~$G$.
In fact, 
our Theorems \ref{T:order p^n}\eqref{I:cyclic of order p^n for n at least 2} 
and~\ref{T:surjection for order 4} 
concerning such $X$ 
are needed for our proof of Theorem~\ref{T:almost rational} 
on almost rational automorphisms of $k[[t]]$.

For some other applications of HKG $G$-curves, e.g., to the problem of lifting
automorphisms of $k[[t]]$ to characteristic~$0$, see \cite{chingurhar} and
its references.

\subsection{Finite-order automorphisms of $k[[t]]$}
\label{s:finiteorder}

Every order~$p$ element of $\Aut(k[[t]])$
is conjugate to $t \mapsto t (1+c t^m)^{-1/m}$
for some $c \in k^\times$ and some positive integer $m$ prime to~$p$
(see~\cite[Proposition~1.2]{klopsch}, \cite[\S4]{lubin}, 
and Theorem~\ref{T:order p}).

The natural question arises whether there is an equally explicit description 
of automorphisms of order $p^n$ for $n>1$.
Each such automorphism is conjugate to 
$t \mapsto \sigma(t)$
for some $\sigma(t) \in k[[t]]$ that is algebraic over $k(t)$
(see Corollary~\ref{C:algebraic series}).
In this case, the field
$L \colonequals k(t,\sigma(t),\ldots,\sigma^{p^n-1}(t)) \subseteq k((t))$
is algebraic over $k(t)$.
When $n>1$, we cannot have $L = k(t)$,
because the group $\Aut_k(k(t)) \isom \PGL_2(k)$ has no element of order $p^2$.
The next simplest case from the point of view of explicit power series
is the following:

\begin{definition}
\label{D:almost rational}
Call $\sigma \in \Aut(k[[t]])$ \defi{almost rational} 
if the field $L \colonequals k(\{\sigma(t): \sigma \in G\})$
is a $\Z/p\Z$ Artin--Schreier extension of $k(t)$;
i.e., $L=k(t,\beta)$ where $\beta \in k((t))$ satisfies
$\wp(\beta) = \alpha$ for some $\alpha \in k(t)$;
here $\wp$ is the Artin--Schreier operator
defined by $\wp(x) \colonequals x^p-x$.
\end{definition}

By subtracting an element of $k[t^{-1}]$ from $\beta$,
we may assume that $\beta \in t k[[t]]$ 
and hence $\alpha \in k(t) \intersect t k[[t]]$.
Then we have an explicit formula for $\beta$, namely
\[
	\beta = - \sum_{i = 0}^\infty \alpha^{p^i},
\]
and $\sigma(t)$ is a rational function in $t$ and $\beta$.
This is the sense in which almost rational automorphisms
have explicit power series.

Prior to the present article, two of us found one explicit example
of an almost rational $\sigma$ of order $p^n > p$ (and its inverse);  see \cite{chinburgsymonds}.
Our first main theorem describes \emph{all} such $\sigma$ up to conjugacy.

\begin{theorem}
\label{T:almost rational}
Suppose that $\sigma$ is an almost rational automorphism
of $k[[t]]$ of order $p^n$ for some $n>1$.
Then $p=2$, $n=2$, and there exists $b \in k$ (unique modulo $\wp(k) = \{\wp(a):a \in k\}$) 
such that $\sigma$ is conjugate to 
the order~$4$ almost rational automorphism
\begin{equation}
\label{E:explicit sigma 1}
	\sigma_b(t) \colonequals \frac{b^2 t + (b+1)t^2 + \beta}{b^2+t^2},
\end{equation}
where $\beta$ 
is the unique solution to $\beta^2-\beta=t^3+(b^2+b+1)t^2$ in $t k[[t]]$.
\end{theorem}

\begin{remark}
If $k$ is algebraically closed, then $\wp(k)=k$, 
so Theorem~\ref{T:almost rational} implies that 
all almost rational automorphisms of order~$4$
lie in one conjugacy class in $\Aut(k[[t]])$.
\end{remark}

\begin{remark}
The  example in \cite{chinburgsymonds} was
\begin{align*}
	\sigma_0(t)
	&=  t + t^2 + \sum_{j = 0}^\infty \sum_{\ell =0 }^{2^{j}-1} t^{6\cdot 2^j+2 \ell} \\
	&= t + t^2 + (t^6) + (t^{12} + t^{14}) + (t^{24} + t^{26} + t^{28} + t^{30}) + \cdots \\
        &= \frac{t}{1+t} + \frac{\gamma}{(1+t)^2}
\end{align*}
over $\F_2$, 
where the series $\gamma \colonequals \sum_{i = 0}^{\infty} (t^3 + t^4)^{2^i}$
satisfies $\gamma^2-\gamma = t^3+t^4$.
(If $\beta$ is as in Theorem~\ref{T:almost rational},
then $\gamma=\beta+t^2$.)
Zieve and Scherr communicated to us that the inverse of $\sigma_0$
has a simpler series, namely 
\[
	\sigma_1(t) = t^{-2} \sum_{i=0}^\infty (t^3+t^4)^{2^i}
	= \sum_{i=0}^\infty t^{3 \cdot 2^i-2} + \sum_{j=2}^\infty t^{2^j-2}.
\]
In general, the inverse of $\sigma_b$ is $\sigma_{b+1}$ (Remark~\ref{R:inverse}).
\end{remark}

\begin{remark}
Let $\sigma$ be any element of finite order in $\Aut(k[[t]])$.
Even if $\sigma$ is not almost rational,
we can assume after conjugation that 
the power series $\sigma(t) = \sum_{i \ge 1} a_i t^i$ 
is algebraic over $k(t)$,
as mentioned above.
When $k$ is finite, this implies that the sequence $(a_i)$
is Turing computable, and even \defi{$p$-automatic};
i.e., there is a finite automaton that calculates $a_i$
when supplied with the base~$p$ expansion of~$i$ \cite{christol, christoletal}.
\end{remark}

\subsection{Harbater--Katz--Gabber $G$-curves}
\label{s:HKGdefs}

An order $p^n$ element of $\Aut(k[[t]])$
induces an injective homomorphism $\Z/p^n\Z \to \Aut(k[[t]])$.
Suppose that we now replace $\Z/p^n\Z$ with any finite group $G$.
Results of Harbater \cite[\S2]{harbater} when $G$ is a $p$-group, 
and of Katz and Gabber \cite[Main Theorem~1.4.1]{katz} in general, 
show that any injective $\alpha \colon G \to \Aut(k[[t]])$
arises from a $G$-action on a curve.
More precisely, $\alpha$ arises from a triple $(X,x,\phi)$ consisting of 
a smooth projective curve $X$,
a point $x \in X(k)$, 
and an injective homomorphism $\phi \colon G \to \Aut(X)$ 
such that $G$ fixes $x$:
here $\alpha$ expresses 
the induced action of $G$ on the completed local ring $\widehat{\OO}_{X,x}$
with respect to some uniformizer $t$.
In Section~\ref{S:HKG G-curves} 
we will define a Harbater--Katz--Gabber $G$-curve (HKG $G$-curve)
to be a triple $(X,x,\phi)$ as above
with $X/G \isom \PP^1_k$ 
such that apart from $x$ there is at most one non-free $G$-orbit,
which is tamely ramified if it exists.
We will sometimes omit $\phi$ from the notation.

HKG $G$-curves 
play a key role in our proof of Theorem~\ref{T:almost rational}.
Our overall strategy is to reduce Theorem~\ref{T:almost rational}
to the classification of certain HKG $G$-curves,
and then to use geometric tools 
such as the Hurwitz formula 
to complete the classification.

\subsection{Harbater--Katz--Gabber $G$-curves with extra automorphisms}
\label{s:extrauts}

In this section, $(X,x)$ is an HKG $G$-curve
and $J$ is a finite group such that $G \le J \le \Aut(X)$.
We do not assume a priori that $J$ fixes $x$.
Let $g_X$ be the genus of~$X$.

\begin{question}
\label{q:bigquestion}
Must $(X,x)$ be an HKG $J$-curve?
\end{question}

The answer is sometimes yes, sometimes no.
Here we state our three main theorems in this direction;
we prove them in Section~\ref{S:extra automorphisms}.

\begin{theorem}
\label{T:J fixes x}
We have that $(X,x)$ is an HKG $J$-curve if and only if $J$ fixes $x$.
\end{theorem}

When $g_X>1$, Theorem~\ref{thm:kgthm} below gives a weaker hypothesis 
that still is sufficient to imply that $(X,x)$ is an HKG $J$-curve.
Let $J_x$ be the decomposition group $\Stab_J(x)$.

\begin{definition}
\label{D:mixitup}
We call the action of $J$ \defi{mixed} 
if there exists $\sigma \in J$ such that $\sigma(x) \ne x$ and 
$\sigma(x)$ is nontrivially 
but tamely ramified with respect to the action of $J_x$,
and \defi{unmixed} otherwise.
\end{definition}

\begin{theorem}
\label{thm:kgthm}
If $g_X>1$ and the action of $J$ is unmixed, 
then $(X,x)$ is an HKG $J$-curve.
\end{theorem}

We will also answer Question~\ref{q:bigquestion} 
in an explicit way when $g_X \le 1$,
whether or not the action of $J$ is mixed.

Finally, if $J$ is solvable, the answer to Question~\ref{q:bigquestion} 
is almost always yes, as the next theorem shows.
For the rest of the paper, $\kbar$ denotes an algebraic closure of~$k$.

\begin{theorem}
\label{T:solvable}
If $J$ is solvable and $(X,x)$ is \emph{not} an HKG $J$-curve,
then one of the following holds:
\begin{itemize}
\item $X \isom \PP^1$;
\item $p$ is $2$ or $3$, and $X$ is an elliptic curve of $j$-invariant $0$;
\item $p=3$, and $X$ is isomorphic over $\kbar$ to 
the genus~$3$ curve $z^4 = t^3 u - t u^3$ in $\PP^2$; or
\item $p=2$, and $X$ is isomorphic over $\kbar$ 
to the smooth projective model of the genus~$10$ affine curve $z^9 = (u^2+u)(u^2+u+1)^3$.
\end{itemize}
\end{theorem}
Each case in Theorem~\ref{T:solvable} actually arises.
We prove a stronger version of Theorem~\ref{T:solvable} in Theorem~\ref{thm:kgsolv} using
the examples discussed in Section~\ref{S:constructions}.

\section{Automorphisms of $k[[t]]$}
\label{s:ARresults}

The purpose of this section is to recall some basic results about $\Aut(k[[t]])$.

\subsection{Groups that are cyclic mod~$p$}

A \defi{$p'$-group} is a finite group of order prime to~$p$.
A finite group $G$ is called \defi{cyclic mod~$p$}
if it has a normal Sylow $p$-subgroup such that the quotient is cyclic.
Equivalently, $G$ is cyclic mod~$p$
if $G$ is a semidirect product $P \rtimes C$ with $P$ a $p$-group
and $C$ a cyclic $p'$-group.
In this case, $P$ is the unique Sylow $p$-subgroup of $G$,
and the Schur--Zassenhaus theorem \cite[Theorem~3.12]{Isaacs2008} 
implies that every subgroup of $G$ isomorphic to $C$ is conjugate to $C$.

\subsection{The Nottingham group}

Any $k$-algebra automorphism $\sigma$ of $k[[t]]$ 
preserves the maximal ideal and its powers,
and hence is $t$-adically continuous,
so $\sigma$ is uniquely determined by specifying 
the power series $\sigma(t) = \sum_{n \ge 1} a_n t^n$ (with $a_1 \in k^\times$).
The map $\Aut(k[[t]]) \to k^\times$ sending $\sigma$ to $a_1$
is a surjective homomorphism.
The \defi{Nottingham group} $\mathcal{N}(k)$ 
is the kernel of this homomorphism;
it consists of the power series $t + \sum_{n \ge 2} a_n t^n$ under composition.
Then $\Aut(k[[t]])$ is a semidirect product $\mathcal{N}(k) \rtimes k^\times$.
For background on $\mathcal{N}(k)$, see, e.g., \cite{camina2}. 

If $k$ is finite, then $\mathcal{N}(k)$ is a pro-$p$ group.
In general, $\mathcal{N}(k)$ is pro-solvable
with a filtration whose quotients are isomorphic to $k$ under addition;
thus every finite subgroup of $\mathcal{N}(k)$ is a $p$-group.
Conversely, Leedham-Green and Weiss, using techniques of Witt, 
showed that any finite $p$-group 
can be embedded in $\mathcal{N}(\F_p)$; 
indeed, so can any countably based pro-$p$ group~\cite{camina}.
The embeddability of finite $p$-groups follows alternatively from 
the fact that the maximal pro-$p$ quotient of 
the absolute Galois group of $k((t^{-1}))$ 
is a free pro-$p$ group of infinite rank \cite[(1.4.4)]{katz}.

On the other hand, any finite subgroup of $k^\times$ is a cyclic $p'$-group.
Thus any finite subgroup of $\Aut(k[[t]])$ is cyclic mod~$p$,
and any finite $p$-group in $\Aut(k[[t]])$ is contained in $\mathcal{N}(k)$.

\subsection{Algebraic automorphisms of $k[[t]]$}
\label{S:algebraic automorphisms}

Call $\sigma \in \Aut(k[[t]])$ \defi{algebraic} if
$\sigma(t)$ is algebraic over $k(t)$.  
\begin{proposition}
\label{P:algebraic automorphisms}
The set $\Aut_{\alg}(k[[t]])$ of all algebraic automorphisms of $k[[t]]$
over $k$ is a subgroup of $\Aut(k[[t]])$.  
\end{proposition}

\begin{proof}
Suppose that $\sigma \in \Aut_{\alg}(k[[t]])$,
so $\sigma(t)$ is algebraic over $k(t)$.
Applying another automorphism $\tau \in \Aut(k[[t]])$ to the algebraic relation 
shows that $\sigma(\tau(t))$ is algebraic over $k(\tau(t))$.
So if $\tau$ is algebraic, so is $\sigma \circ \tau$. 
On the other hand, taking $\tau = \sigma^{-1}$ 
shows that $t$ is algebraic over $k(\sigma^{-1}(t))$.
Since $t$ is not algebraic over~$k$, 
this implies that $\sigma^{-1}(t)$ is algebraic over $k(t)$.
\end{proof}

\subsection{Automorphisms of order $p$}
\label{S:automorphisms of order p}

The following theorem was proved by Klopsch~\cite[Proposition~1.2]{klopsch}
and reproved by Lubin~\cite[\S4]{lubin}
(they assumed that $k$ was finite, but this is not crucial).
Over algebraically closed fields it was shown in \cite[p.~211]{BertinMezard} by
 Bertin and M\'{e}zard, who mention related work of
Oort, Sekiguchi and Suwa in \cite{OSS}.
For completeness, we give here a short proof, similar to the proofs 
in \cite[Appendix]{klopsch} and \cite[p.~211]{BertinMezard}; 
it works over any perfect field $k$ of characteristic $p>0$.

\begin{theorem}
\label{T:order p}
Every $\sigma \in \mathcal{N}(k)$ of order~$p$
is conjugate in $\mathcal{N}(k)$ to $t \mapsto t (1+c t^m)^{-1/m}$
for a unique positive integer $m$ prime to~$p$
and a unique $c \in k^\times$.
The automorphisms given by $(m,c)$ and $(m',c')$ 
are conjugate in $\Aut(k[[t]])$
if and only if $m=m'$ and $c/c' \in k^{\times m}$.
\end{theorem}

\begin{proof}
Extend $\sigma$ to the fraction field $k((t))$.
By Artin--Schreier theory, there exists $y \in k((t))$
such that $\sigma(y)=y+1$.
This $y$ is unique modulo $k((t))^\sigma$.
Since $\sigma$ acts trivially on the residue field of $k[[t]]$, 
we have $y \notin k[[t]]$.
Thus $y = c t^{-m} + \cdots$ for some $m \in \Z_{>0}$ and $c \in k^\times$.
Choose $y$ so that $m$ is minimal.
If the ramification index $p$ divided $m$, 
then we could subtract from $y$ an element of $k((t))^\sigma$ 
with the same leading term,
contradicting the minimality of $m$.
Thus $p\nmid m$.
By Hensel's lemma, $y = c (t')^{-m}$ for some $t' = t + \cdots$.
Conjugating by the automorphism $t \mapsto t'$
lets us assume instead that $y = c t^{-m}$.
Substituting this into $\sigma(y)=y+1$ yields
$c \, \sigma(t)^{-m} = c t^{-m} + 1$.
Equivalently, $\sigma(t) = t (1 + c^{-1} t^m)^{-1/m}$.
Rename $c^{-1}$ as $c$.

Although $y$ is determined only modulo $\wp(k((t)))$,
the leading term of a minimal $y$ is determined.
Conjugating $\sigma$ in $\Aut(k[[t]])$
amounts to expressing $\sigma$ with respect to
a new uniformizer $u = u_1 t + u_2 t^2 + \cdots$.
This does not change $m$,
but it multiplies $c$ by $u_1^m$.
Conjugating $\sigma$ in $\mathcal{N}(k)$
has the same effect, except that $u_1=1$,
so $c$ is unchanged too.
\end{proof}

\begin{remark}
For each positive integer $m$ prime to $p$,
let $\Disp_m \colon \mathcal{N}(k) \to \mathcal{N}(k)$ 
be the map sending $t \mapsto f(t)$ to $t \mapsto f(t^m)^{1/m}$
(we take the $m$th root of the form $t+\cdots$).
This is an injective endomorphism of the group $\mathcal{N}(k)$,
called \defi{$m$-dispersal} in \cite{lubin}.
It would be conjugation by $t \mapsto t^m$,
except that $t \mapsto t^m$ is not in $\Aut(k[[t]])$ (for $m>1$).
The automorphisms in Theorem~\ref{T:order p}
may be obtained from $t \mapsto t(1+t)^{-1}$
by conjugating by $t \mapsto ct$ and then dispersing.
\end{remark}

\section{Ramification and the Hurwitz Formula}
\label{sec:hurwitz}

Here we review the Hurwitz formula and related facts we need later.

\subsection{Notation}

By a \defi{curve} over $k$ 
we mean a $1$-dimensional smooth projective geometrically integral scheme $X$ 
of finite type over $k$.
For a curve $X$, let $k(X)$ denote its function field,
and let $g_X$ or $g_{k(X)}$ denote its genus.
If $G$ is a finite group acting on a curve $X$,
then $X/G$ denotes the curve 
whose function field is the invariant subfield $k(X)^G$.

\subsection{The local different}
\label{S:local different}

Let $G$ be a finite subgroup of $\Aut(k[[t]])$.
For $i \ge 0$, define the \defi{ramification subgroup} 
$G_i \colonequals 
\{ g \in G \mid g \mbox{ acts trivially on } k[[t]]/(t^{i+1}) \}$
as usual.
Let $\mathfrak{d}(G) \colonequals \sum_{i=0}^\infty (|G_i|-1) \in \Z_{\ge 0}$;
this is the exponent of 
the \defi{local different} \cite[IV, Proposition~4]{serre}.

\subsection{The Hurwitz formula}

In this paragraph we assume that $k$ is an algebraically closed field of characteristic $p>0$.
Let $H$ be a finite group acting faithfully on a curve $X$ over~$k$.
For each $s \in X(k)$, let $H_s \le H$ be the inertia group.
We may identify $\widehat{\OO}_{X,s}$ with $k[[t]]$
and $H_s$ with a finite subgroup $G \le \Aut(k[[t]])$;
then define $\mathfrak{d}_s = \mathfrak{d}_s(H) \colonequals \mathfrak{d}(H_s)$.
We have $\mathfrak{d}_s>0$ if and only if $s$ is ramified.
If $s$ is tamely ramified, meaning that $H_s$ is a $p'$-group, 
then $\mathfrak{d}_s = |H_s|-1$.
The Hurwitz formula \cite[IV, 2.4]{hartshorne} is
\[
	2g_X-2 = |H|(2g_{X/H}-2) + \sum_{s \in X(k)} \mathfrak{d}_s.
\]
\begin{remark}
\label{rem:hurwitz}
When we apply the Hurwitz formula to a curve over a perfect field 
that is not algebraically closed, 
it is understood that we first extend scalars to an algebraic closure.
\end{remark}

\subsection{Lower bound on the different}
\label{S:breaks}

We continue to assume that $k$ is an algebraically closed field of characteristic $p>0$. 
The following material is taken from \cite[IV]{serre}, 
as interpreted by Lubin in \cite{lubin}.  
Let $G$ and the $G_i$ be as in Section~\ref{S:local different}.
An integer $i \ge 0$ is a \defi{break} in the lower numbering
of the ramification groups of $G$ if $G_i \ne G_{i+1}$.
Let $b_0$, $b_1$, \dots be the breaks in increasing order;
they are all congruent modulo $p$.
The group $G_0/G_1$ embeds into $k^\times$, while $G_i/G_{i+1}$ embeds in the additive group of $k$ if $i \ge 1$.

{}From now on, assume that $G$ is a cyclic group of order $p^n$ with generator $\sigma$. Then $G_0 =  G_1$ and  
each quotient $G_i/G_{i+1}$ is killed by $p$.
Thus there must be exactly $n$ breaks $b_0,\ldots,b_{n-1}$.
If $0 \le i \le b_0$, then $G_i = G$;
if $1 \le j \le n -1$ and $b_{j -1} < i \le b_j$,
then $|G_i| = p^{n-j}$;
and if $b_{n-1} < i$, then $G_i = \{e\}$.
According to the Hasse--Arf theorem, 
there exist positive integers $i_0, \ldots, i_{n-1}$ 
such that $b_j= i_0 +pi_1 + \cdots +p^ji_j$ for $0 \le j \le n-1$. 
Then 
\begin{equation}
\label{E:different formula}
	\mathfrak{d}(G)= (i_0+1)(p^n-1)+i_1(p^n-p) + \cdots +i_{n-1}(p^n-p^{n-1}).
\end{equation}
The upper breaks $b^{(j)}$ we do not need to define here, 
but they have the property that in the cyclic case, 
$b^{(j)}=i_0+ \cdots + i_j$ for $0 \le j \le n-1$. 

Local class field theory shows 
that $p \nmid b^{(0)}$, 
that $b^{(j)} \geq pb^{(j-1)}$ for $1 \le j \le n-1$, 
and that if this inequality is strict then $p \nmid b^{(j)}$; 
this is proved in \cite[XV,~\S2~Thm.~2]{serre} for quasi-finite residue fields, 
and extended to algebraically closed residue fields 
in \cite[Prop. 13.2]{chingurhar}.
Conversely, any sequence of positive numbers $b^{(0)},\ldots,b^{(n-1)}$ 
that satisfies these three conditions 
is realized by some element of order $p^n$ 
in $\Aut(k[[t]])$ \cite[Observation~5]{lubin}.

Thus $i_0 \geq 1$, and $i_j \geq (p-1)p^{j-1}$ for $1 \le j \le n-1$.  
Substituting into~\eqref{E:different formula} yields the following result.

\begin{lemma}
\label{L:lower bound on different}
If $G$ is cyclic of order $p^n$, then
\[
	\mathfrak{d}(G) \geq \frac{p^{2n}+p^{n+1}+p^n-p-2}{p+1}
\]
and this bound is sharp. 
\end{lemma}

\begin{remark}
\label{rem:different}
Lemma~\ref{L:lower bound on different}
is valid over any perfect field $k$ of characteristic~$p$,
because extending scalars to $\kbar$ does not change $\mathfrak{d}(G)$.
\end{remark}

\section{Harbater--Katz--Gabber $G$-curves}
\label{s:KGdefs}

Let $k$ be a perfect field of characteristic $p>0$.

\subsection{Pointed $G$-curves}
\label{S:pointed G-curves}

\begin{definition}
A \defi{pointed $G$-curve} over $k$ is a triple $(X,x,\phi)$ consisting of
a curve $X$, a point $x \in X(k)$, 
and an injective homomorphism $\phi \colon G \to \Aut(X)$
such that $G$ fixes $x$.
(We will sometimes omit $\phi$ from the notation.)
\end{definition}

Suppose that $(X,x,\phi)$ is a pointed $G$-curve.
The faithful action of $G$ on $X$ induces a faithful action on $k(X)$.
Since $G$ fixes $x$, the latter action induces 
a $G$-action on the $k$-algebras $\OO_{X,x}$ and $\widehat{\OO}_{X,x}$.
Since $\Frac(\OO_{X,x}) = k(X)$ and $\OO_{X,x} \subseteq \widehat{\OO}_{X,x}$,
the $G$-action on $\widehat{\OO}_{X,x}$ is faithful too.
Since $x \in X(k)$, a choice of uniformizer $t$ at $x$
gives a $k$-isomorphism $\widehat{\OO}_{X,x} \iso k[[t]]$.
Thus we obtain an embedding $\rho_{X,x,\phi} \colon G \injects \Aut(k[[t]])$.
Changing the isomorphism $\widehat{\OO}_{X,x} \isom k[[t]]$
conjugates $\rho_{X,x,\phi}$ by an element of $\Aut(k[[t]])$,
so we obtain a map
\begin{align}
\label{E:take action at infinity}
	\{\textup{pointed $G$-curves}\} &\to \{\textup{conjugacy classes of embeddings $G \injects \Aut(k[[t]])$}\} \\
\nonumber	(X,x,\phi) &\mto [\rho_{X,x,\phi}].
\end{align}
Also, $G$ is the inertia group of $X \to X/G$ at $x$.

\begin{lemma}
\label{L:G is cyclic mod p}
If $(X,x,\phi)$ is a pointed $G$-curve, then $G$ is cyclic mod~$p$.
\end{lemma}

\begin{proof}
The group $G$ is embedded as a finite subgroup of $\Aut(k[[t]])$.
\end{proof}

\subsection{Harbater--Katz--Gabber $G$-curves}
\label{S:HKG G-curves}

\begin{definition}
\label{dfn:katzG} 
A pointed $G$-curve $(X,x,\phi)$ over $k$ is called a 
\defi{Harbater--Katz--Gabber $G$-curve (HKG $G$-curve)} 
if both of the following conditions hold:
\begin{enumerate}[\upshape (i)]
\item\label{I:quotient is genus 0} 
The quotient $X/G$ is of genus $0$.
(This is equivalent to $X/G \isom \PP^1_k$,
since $x$ maps to a $k$-point of $X/G$.)
\item\label{I:almost unramified}
The action of $G$ on $X-\{x\}$ is either unramified everywhere,
or tamely and nontrivially ramified at one $G$-orbit in $X(\kbar)-\{x\}$ 
and unramified everywhere else. 
\end{enumerate} 
\end{definition}

\begin{remark}
\label{R:katz}
Katz in~\cite[Main Theorem~1.4.1]{katz} 
focused on the \emph{base curve} $X/G$ as starting curve.
He fixed an isomorphism of $X/G$ with $\PP^1_k$ identifying
the image of $x$ with $\infty$ and 
the image of a tamely and nontrivially ramified point of $X(\kbar) - \{x\}$
(if such exists) with $0$.
He then considered Galois covers $X \to X/G = \PP^1_k$ 
satisfying properties as above;
these were called Katz--Gabber covers in~\cite{chingurhar}.
For our applications, however, 
it is more natural to focus on the upper curve $X$.
\end{remark}

HKG curves have some good functoriality properties 
that follow directly from the definition:
\begin{itemize} 
\item
\emph{Base change:}
Let $X$ be a curve over $k$, let $x \in X(k)$,
and let $\phi \colon G \to \Aut(X)$ be a homomorphism.
Let $k' \supseteq k$ be a field extension.
Then $(X,x,\phi)$ is an HKG $G$-curve over $k$ 
if and only if its base change to $k'$ is an HKG $G$-curve over $k'$.
\item
\emph{Quotient:} 
If $(X,x,\phi)$ is an HKG $G$-curve,
and $H$ is a normal subgroup of $G$,
then $X/H$ equipped with the image of $x$ and the induced $G/H$-action
is an HKG $G/H$-curve.
\end{itemize}

\begin{example}
\label{Ex:p-curve}
Let $P$ be a finite subgroup of the additive group of $k$,
so $P$ is an elementary abelian $p$-group.
Then the addition action of $P$ on $\Aff^1_k$
extends to an action $\phi \colon P \to \Aut(\PP^1_k)$
totally ramified at $\infty$ and unramified elsewhere,
so $(\PP^1_k,\infty,\phi)$ is an HKG $P$-curve.
\end{example}

\begin{example}
\label{Ex:p'-curve}
Suppose that $C$ is a $p'$-group and that $(X,x,\phi)$ is an HKG $C$-curve.  
By Lemma~\ref{L:G is cyclic mod p}, $C$ is cyclic. 
By the Hurwitz formula,
$X$ must have genus $0$ since there are at most
two $C$-orbits of ramified points and all the ramification is tame.  
Moreover, $X$ has a $k$-point (namely, $x$), so $X \isom \PP^1_k$, 
and $C$ is a $p'$-subgroup of the stabilizer of $x$ inside
$\Aut(X) \isom \Aut(\PP^1_k) \isom \PGL_2(k)$.  It follows
that after applying an automorphism of $X = \PP^1_k$, we can assume that $C$
fixes the points $0$ and $\infty$ and corresponds to the multiplication action 
of a finite subgroup of $k^\times$ on $\Aff^1_k$.  Conversely, such an action 
gives rise to an HKG $C$-curve $(\PP^1_k,\infty,\phi)$.
\end{example}

The following gives alternative criteria for testing
whether a pointed $G$-curve is an HKG $G$-curve.

\begin{proposition}
\label{P:G-curve vs. P-curve}
Let $(X,x,\phi)$ be a pointed $G$-curve.
Let $P$ be the Sylow $p$-subgroup of $G$.
Then the following are equivalent:
\begin{enumerate}[\upshape (i)]
\item\label{I:G-curve} $(X,x,\phi)$ is an HKG $G$-curve.
\item\label{I:P-curve}  $(X,x,\phi|_P)$ is an HKG $P$-curve.
\item\label{I:strong P-curve} The quotient $X/P$ is of genus $0$,
and the action of $P$ on $X-\{x\}$ is unramified.
\item\label{I:genus of HKG curve}
Equality holds in the inequality $g_X \ge 1 -|P| + \mathfrak{d}_x(P)/2$.
\end{enumerate}
\end{proposition}

\begin{proof}
Let $C=G/P$.

\eqref{I:strong P-curve}$\Rightarrow$\eqref{I:P-curve}: 
Trivial.

\eqref{I:G-curve}$\Rightarrow$\eqref{I:strong P-curve}: 
By the quotient property of HKG curves, $X/P$ is an HKG $C$-curve,
so $X/P \isom \PP^1_k$ by Example~\ref{Ex:p'-curve}.
At each $y \in X(\kbar)-\{x\}$, the ramification index $e_y$ 
for the $P$-action divides $|P|$ but is prime to $p$, so $e_y=1$.
Thus the action of $P$ on $X-\{x\}$ is unramified.

\eqref{I:P-curve}$\Rightarrow$\eqref{I:G-curve}: 
Applying the result \eqref{I:G-curve}$\Rightarrow$\eqref{I:strong P-curve}
to $P$ shows that $X \to X/P$ is unramified outside $x$.
There is a covering $\PP^1_k \isom X/P \to X/G$,
so $X/G \isom \PP^1_k$.
We may assume that $C \ne \{1\}$.
By Example~\ref{Ex:p'-curve}, 
the cover $X/P \to X/G$ is totally tamely ramified above two $k$-points,
and unramified elsewhere.
One of the two points must be the image of $x$;
the other is the image of the unique tamely ramified $G$-orbit in $X(\kbar)$,
since $X \to X/P$ is unramified outside~$x$.

\eqref{I:strong P-curve}$\Leftrightarrow$\eqref{I:genus of HKG curve}:
The Hurwitz formula (see Remark~\ref{rem:hurwitz}) for the action of $P$
simplifies to the inequality in \eqref{I:genus of HKG curve}
if we use $g_{X/P} \ge 0$ and discard ramification in $X-\{x\}$.
Thus \emph{equality} holds in \eqref{I:genus of HKG curve}
if and only if $g_{X/P}=0$ and the action of $P$ on $X - \{x\}$ is unramified.
\end{proof}

\subsection{The Harbater--Katz--Gabber theorem}
\label{s:HKGtheorem}

The following is a consequence of work of Harbater \cite[\S 2]{harbater} when $G$ is a $p$-group and of  
Katz and Gabber \cite[Main Theorem~1.4.1]{katz} when $G$ is arbitrary. 

\begin{theorem}[Harbater, Katz--Gabber]
\label{thm:KatzGstate}
The assignment $(X,x,\phi) \mapsto \rho_{X,x,\phi}$ induces a surjection from  the set of HKG $G$-curves over $k$ up to equivariant isomorphism to the set of conjugacy classes of embeddings of $G$ into $\Aut(k[[t]])$.\end{theorem}

\begin{corollary}
\label{cor:field}
Any finite subgroup of $\Aut_{\kbar}(\kbar[[t]])$ 
can be conjugated into $\Aut_{k'}(k'[[t]])$ 
for some finite extension $k'$ of $k$ in $\kbar$.
\end{corollary}

\begin{proof}
The subgroup is realized by some HKG curve over $\kbar$.
Any such curve is defined over some finite extension $k'$ of $k$.
\end{proof}

\begin{corollary}
\label{C:algebraic series}
Any finite subgroup of $\Aut(k[[t]])$ 
can be conjugated into $\Aut_{\alg}(k[[t]])$.
\end{corollary}

\begin{proof}
The subgroup is realized by some HKG curve $X$.
By conjugating, we may assume that the uniformizer $t$ 
is a rational function on $X$.
Then each power series $\sigma(t)$ 
represents another rational function on $X$,
so $\sigma(t)$ is algebraic over $k(t)$.
\end{proof}

\section{Almost rational automorphisms}
\label{s:ArA}

\subsection{The field generated by a group of algebraic automorphisms}

Let $G$ be a finite subgroup of $\Aut_{\alg}(k[[t]])$.
Let $L \colonequals k(\{\sigma(t): \sigma \in G\}) \subseteq k((t))$.
Then $L$ is a finite extension of $k(t)$,
so $L \isom k(X)$ for some curve $X$.
The $t$-adic valuation on $k((t))$ restricts to a valuation on $L$
associated to a point $x \in X(k)$.
The $G$-action on $k((t))$ preserves $L$.  
This induces an embedding $\phi \colon G \to \Aut(X)$ 
such that $G$ fixes $x$, 
so $(X,x,\phi)$ is a pointed $G$-curve over $k$.

\begin{theorem}
\label{T:order p^n}
Let $G$ be a finite subgroup of $\Aut_{\alg}(k[[t]])$.
Let $L$ and $(X,x,\phi)$ be as above.
Let $d \colonequals [L:k(t)]$.
\begin{enumerate}[\upshape (a)]
\item \label{I:d-controlled}
We have $g_X \le (d-1)^2$.
\item \label{I:cyclic of order p^n for n at least 1}
If $G$ is cyclic of order $p^n$, 
then $g_X \ge \dfrac{p(p^n-1)(p^{n-1}-1)}{2(p+1)}$.
Moreover, if equality holds, then $(X,x,\phi)$ is an HKG $G$-curve.
\item  \label{I:cyclic of order p^n for n at least 2}
Suppose that $G$ is cyclic of order $p^n$.  Then
\begin{equation}
\label{eq:tautologicalbound}
d \ge 1 + \sqrt{\frac{p(p^n-1)(p^{n-1} - 1)}{2(p+1)}}.
\end{equation}
In particular, if $d \le p$ and $n \ge 2$, 
then $d=p=n=2$ and $(X,x,\phi)$ is an HKG $\Z/4\Z$-curve of genus~$1$.
\end{enumerate}
\end{theorem}

\begin{proof}\hfill
\begin{enumerate}[\upshape (a)]
\item 
In \cite[\S2]{Poonen-gonality},
a subfield $F \subseteq L$ is called \defi{$d$-controlled} 
if there exists $e \in \Z_{>0}$ 
such that $[L:F] \le d/e$ and $g_F \le (e-1)^2$.
In our setting, the $G$-action on $k((t))$ preserves $L$,
so $[L:k(\sigma(t))]=d$ for every $\sigma \in G$.
By \cite[Corollary~2.2]{Poonen-gonality}, 
$L \subseteq L$ is $d$-controlled.
Here $d/e=1$, so $g_L \le (e-1)^2 = (d-1)^2$.
\item 
In the inequality $g_X \ge 1 - |G| + \mathfrak{d}_x(G)/2$
of Proposition~\ref{P:G-curve vs. P-curve}\eqref{I:genus of HKG curve},
substitute $|G|=p^n$ and the bound of Lemma~\ref{L:lower bound on different}.
If equality holds, then 
Proposition~\ref{P:G-curve vs. P-curve}\eqref{I:genus of HKG curve}$\Rightarrow$\eqref{I:G-curve} 
shows that $(X,x,\phi)$ is an HKG $G$-curve.
\item 
Combine the upper and lower bounds on $g_X$ in \eqref{I:d-controlled} 
and~\eqref{I:cyclic of order p^n for n at least 1}.  
If $d \le p$ and $n \ge 2$, then 
\[
	p \ge d \ge 1 + \sqrt{\frac{p(p^2-1)(p-1)}{2(p+1)}} = 1 + (p-1) \sqrt{\frac{p}{2}} \ge 1+(p-1) = p,
\]
so equality holds everywhere.
In particular, $p=d$, $n=2$, and $p/2=1$, 
so $d=p=n=2$.
Also, \eqref{I:cyclic of order p^n for n at least 1}
shows that $(X,x,\phi)$ is an HKG $G$-curve,
and $g_X=(d-1)^2=1$.
\qedhere
\end{enumerate}
\end{proof}

\begin{remark} 
Part~\eqref{I:cyclic of order p^n for n at least 2} of Theorem~\ref{T:order p^n} 
implies the first statement in Theorem~\ref{T:almost rational}, 
namely that if $\sigma$ is an almost rational
automorphism of order $p^n > p$, then $p = n = 2$. 
To complete the proof of Theorem~\ref{T:almost rational} we will classify in 
 Section~\ref{S:almost rational of order 4} the $\sigma$ when $p = n = 2$. 
\end{remark}

\subsection{Almost rational automorphisms of order $4$}
\label{S:almost rational of order 4}

In this section, $k$ is a perfect field of characteristic~$2$,
and $G=\Z/4\Z$.

\begin{definition}
\label{D:E_ab}
For $a,b \in k$, 
let $E_{a,b}$ be the projective closure of
\[
	z^2 - z = w^3 + (b^2+b+1) w^2 + a.
\]
Let $O \in E_{a,b}(k)$ be the point at infinity,
and let $\phi \colon \Z/4\Z \to \Aut(E_{a,b})$
send $1$ to the order~$4$ automorphism
\[
	\sigma \colon (w,z) \longmapsto (w+1,z+w+b).
\]
\end{definition}

\begin{proposition}
\label{P:E_ab}
Each $(E_{a,b},O,\phi)$ in Definition~\ref{D:E_ab} 
is an HKG $\Z/4\Z$-curve over $k$.
\end{proposition}

\begin{proof}
The automorphism $\sigma$ fixes $O$.
Also, $\sigma^2$ maps $(w,z)$ to $(w,z+1)$,
so $\sigma^2$ fixes only $O$; hence the $G$-action on $E_{a,b} - \{O\}$
is unramified.
Since $E_{a,b} \to E_{a,b}/G$ is ramified, the genus of $E_{a,b}/G$ is $0$.
\end{proof}

\begin{proposition}
\label{P:Z/4Z curves}
Let $k$ be a perfect field of characteristic~$2$.
Let $G=\Z/4\Z$.
For an HKG $G$-curve $(X,x,\phi')$ over $k$, the following are equivalent:
\begin{enumerate}[\upshape (i)]
\item\label{I:Z/4Z curves genus 1}
The genus of $X$ is $1$.
\item\label{I:Z/4Z curves G_i}
The lower ramification groups for $X \to X/G$ at $x$
satisfy $|G_0|=|G_1|=4$, $|G_2|=|G_3|=2$, and $|G_i|=1$ for $i \ge 4$.
\item\label{I:Z/4Z curves G_4}
The ramification group $G_4$ equals $\{1\}$.
\item\label{I:Z/4Z curves E_ab}
There exist $a,b \in k$ such that 
$(X,x,\phi')$ is isomorphic to the HKG $G$-curve $(E_{a,b},O,\phi)$
of Definition~\ref{D:E_ab}.
\end{enumerate}
\end{proposition}

\begin{proof}
Let $g$ be the genus of $X$.
Since $G$ is a $2$-group, $|G_0|=|G_1|=4$.

\eqref{I:Z/4Z curves G_i}$\Rightarrow$\eqref{I:Z/4Z curves genus 1}:
This follows from the Hurwitz formula (see Remark~\ref{rem:hurwitz})
\[
	2g-2 = 4(-2) + \sum_{i \ge 0} (|G_i|-1).
\]

\eqref{I:Z/4Z curves genus 1}$\Rightarrow$\eqref{I:Z/4Z curves G_i}:
If $g=1$, then the Hurwitz formula 
yields $0 = -8 + 3 + 3 + \sum_{i \ge 2} (|G_i|-1)$.
Since the $|G_i|$ form a decreasing sequence of powers of $2$
and include all the numbers $4$, $2$, and $1$ (see Section~\ref{S:breaks}),
the only possibility is as in~\eqref{I:Z/4Z curves G_i}.

\eqref{I:Z/4Z curves G_i}$\Rightarrow$\eqref{I:Z/4Z curves G_4}: Trivial.

\eqref{I:Z/4Z curves G_4}$\Rightarrow$\eqref{I:Z/4Z curves G_i}: 
The lower breaks (see Section~\ref{S:breaks}) satisfy $1 \leq b_0 < b_1 < 4$. 
Since $b_0 \equiv b_1 \pmod{2}$, \eqref{I:Z/4Z curves G_i} follows.

\eqref{I:Z/4Z curves E_ab}$\Rightarrow$\eqref{I:Z/4Z curves genus 1}:
The formulas in \cite[III.\S 1]{silverman} 
show that $E_{a,b}$ is an elliptic curve, hence of genus~$1$.

\eqref{I:Z/4Z curves genus 1}$\Rightarrow$\eqref{I:Z/4Z curves E_ab}:
By \cite[A.1.2(c)]{silverman},
an elliptic curve with an order~$4$ automorphism 
has $j$-invariant $1728=0 \in k$.
By \cite[A.1.1(c)]{silverman},
it has an equation $y^2 + a_3 y = x^3 + a_4 x + a_6$.
Substituting $y \mapsto y + a_3^{-1} a_4 x$ leads to an alternative form
$y^2 + a_3 y = x^3 + a_2 x^2 + a$.
Let $u \in k^\times$ be such that $\sigma^*$ acts on $H^0(X,\Omega^1)$
by multiplication by $u^{-1}$.
Then $u^4=1$, so $u=1$.
By \cite[p.~49]{silverman},
$\sigma$ has the form $(x,y) \mapsto (x+r,y+sx+t)$
for some $r,s,t \in k$.
Since $\sigma^2 \ne 1$, we have $s \ne 0$.
Conjugating by a change of variable $(x,y) \mapsto (\epsilon^2 x,\epsilon^3 y)$ 
lets us assume that $s=1$.
The condition that $(x,y) \mapsto (x+r,y+x+t)$
preserves $y^2 + a_3 y = x^3 + a_2 x^2 + a$
implies that $a_3=r=1$ and $a_2 = t^2+t+1$.
Rename $t,x,y$ as $b,w,z$.
\end{proof}

\begin{corollary}
The HKG $\Z/4\Z$-curves that are minimally ramified
in the sense of having the smallest value of $\inf\{i: G_i = \{1\}\}$
are those satisfying the equivalent conditions 
in Proposition~\ref{P:Z/4Z curves}.
\end{corollary}

Let $\wp(x) \colonequals x^2-x$ be the Artin--Schreier operator 
in characteristic~$2$.  The following lemma is clear.

\begin{lemma}
\label{L:little Artin-Schreier}
Let $L/K$ be a $\Z/2\Z$ Artin--Schreier extension,
so there exist $a \in K$ and $b \in L-K$ such that 
$\wp(b) = a$.
If $x \in L-K$ satisfies $\wp(x) \in K$, then $x \in b + K$.
\end{lemma}

\begin{theorem}
\label{T:surjection for order 4}
Let $k$ be a perfect field of characteristic~$2$.
Let $G=\Z/4\Z$.  Let $\XX$ be the set of HKG $G$-curves
satisfying the equivalent conditions in Proposition~\ref{P:Z/4Z curves}.
Then
\begin{enumerate}[\upshape (a)]
\item\label{I:surjection from XX}
The map~\eqref{E:take action at infinity}
restricts to a surjection from $\XX$
to the set of conjugacy classes in $\Aut(k[[t]])$
containing an almost rational automorphism of order~$4$.
\item\label{I:explicit sigma}
Explicitly, $E_{a,b}$ 
(made into an HKG $G$-curve as in Proposition~\ref{P:E_ab})
maps to the conjugacy class of 
\begin{equation}
\label{E:explicit sigma}
	\sigma_b(t) \colonequals \frac{b^2 t + (b+1)t^2 + \beta}{b^2+t^2},
\end{equation}
where $\beta \colonequals \sum_{i=0}^\infty (t^3+(b^2+b+1)t^2)^{2^i}$
is the unique solution to $\beta^2-\beta=t^3+(b^2+b+1)t^2$ in $t k[[t]]$.
\item\label{I:E_ab and E_a'b'}
For $b,b' \in k$, the automorphisms $\sigma_b,\sigma_{b'} \in \Aut(k[[t]])$
are conjugate if and only if $b \equiv b' \pmod{\wp(k)}$.
\end{enumerate}
\end{theorem}

\begin{proof}\hfill

\eqref{I:surjection from XX} 
First we show that each $E_{0,b}$ maps to a conjugacy class containing
an almost rational automorphism;
the same will follow for $E_{a,b}$ for $a \ne 0$
once we show in the proof of~\eqref{I:E_ab and E_a'b'}
that $E_{a,b}$ gives rise to the same conjugacy class as $E_{0,b}$.
Let $P \colonequals (0,0) \in E_{0,b}(k)$.
Composing $w$ with translation-by-$P$ 
yields a new rational function $w_P = z/w^2$ on $E_{0,b}$;
define $z_P$ similarly, so $z_P = 1 - z^2/w^3$.
Since $w$ has a simple zero at $P$,
the function $t \colonequals w_P$ has a simple zero at $O$.
Also, $\sigma^j(t) \in k(E_{0,b}) = k(t,z_P)$,
which shows that $\sigma$ is almost rational
since $z_P^2 - z_P = w_P^3 + (b^2+b+1) w_P^2$.

Now suppose that $\sigma$ is 
any almost rational automorphism of order~$4$.
Theorem~\ref{T:order p^n}\eqref{I:cyclic of order p^n for n at least 2} 
shows that $\sigma$ arises from an HKG $\Z/4\Z$-curve of genus~$1$,
i.e., a curve as in
Proposition~\ref{P:Z/4Z curves}\eqref{I:Z/4Z curves genus 1}.

\eqref{I:explicit sigma}
Again by referring to the proof of~\eqref{I:E_ab and E_a'b'},
we may assume $a=0$.
Follow the first half of the proof of~\eqref{I:surjection from XX} 
for $E_{0,b}$.
In terms of the translated coordinates $(w_P,z_P)$ on $E_{0,b}$,
the order~$4$ automorphism of the elliptic curve is
\[
	(t,\beta) \longmapsto \sigma((t,\beta) - P) + P.
\]
It is a straightforward but lengthy exercise to show that 
the first coordinate equals 
the expression $\sigma_b(t)$ in (\ref{E:explicit sigma}).  
One uses $t = w_P = z/w^2$,
$\beta = z_P = 1 - z^2/w^3$, 
and the formulas $\sigma(w) = w+1$ and $\sigma(z) = z + w + b$.
In verifying equalities in the field $k(t,\beta)$, one can use the fact that 
$k(t,\beta)$ is the quadratic Artin--Schreier extension of $k(t)$ defined
by $\beta^2 - \beta = t^3 + (b^2 + b + 1)t^2$.

\eqref{I:E_ab and E_a'b'} 
Let $v \colonequals w^2-w$.
Let $\widehat{\OO}$ be the completion of the local ring of $E_{a,b}$ 
at the point $O$ at infinity,
and let $\widehat{K} \colonequals \Frac(\widehat{\OO}) = k((w^{-1}))(z^{-1})$.
With respect to the discrete valuation on $\widehat{K}$, 
the valuations of $w$, $z$ and $v$ are $-2$, $-3$ and $-4$, respectively.  
With respect to the discrete valuation on $k((w^{-1}))$, 
the valuation of $w$ is $-1$ and the valuation of $v$ is $-2$.  
We have $\widehat{K}^G = k((v^{-1}))$.
Define $w'$, $z'$, $v'$, $\sigma'$, $\widehat{\OO}'$, 
and $\widehat{K}' = k(({w'}^{-1}))({z'}^{-1})$ similarly for $E_{a',b'}$.
By definition of the map~\eqref{E:take action at infinity},
$E_{a,b}$ and $E_{a',b'}$ give rise to the same conjugacy class 
if and only if
there exists a $G$-equivariant continuous isomorphism
$\widehat{\OO} \isomto \widehat{\OO}'$
or equivalently $\alpha \colon \widehat{K} \isomto \widehat{K}'$.
It remains to prove that $\alpha$ exists 
if and only if $b \equiv b' \pmod{\wp(k)}$.

$\implies$:
Suppose that $\alpha$ exists.
Lemma~\ref{L:little Artin-Schreier}
shows that $\alpha(w) = w' + f$ for some $f \in k(({v'}^{-1}))$.
Since $\alpha$ preserves valuations, $f \in k[[{v'}^{-1}]]$.
Since $v'$ has valuation $-2$ in $k(({w'}^{-1}))$, the valuation of ${v'}^{-1}$
in this field is $2$.  Therefore $f \in k[[{v'}^{-1}]]$ implies 
 $f = c + \sum_{i \ge 2} f_i {w'}^{-i}$ for some $c, f_i \in k$.
Similarly, $\alpha(z) = z' + h$ 
for some $h = \sum_{i \ge -1} h_i {w'}^{-i} \in w' k[[{w'}^{-1}]]$.
Subtracting the equations
\begin{align*}
	\alpha(z)^2 - \alpha(z) &= \alpha(w)^3 + (b^2+b+1) \alpha(w)^2 + a \\
	{z'}^2 - z' &= {w'}^3 + ({b'}^2+b'+1) {w'}^2 + a' 
\end{align*}
yields
\begin{align}
\label{E:h and f}
	h^2 - h &= (w' + f)^3 - {w'}^3 + (b^2+b+1) (w'+f)^2 - ({b'}^2+b'+1) {w'}^2 + a - a' \\
\nonumber	 &= {w'}^2 f + w' f^2 + f^3 + \wp(b-b') {w'}^2 + (b^2+b+1) f^2 +a-a' \\
\label{E:h and f congruence}
 h^2-h	&\equiv (c + \wp(b-b')) {w'}^2 + c^2 w' + (f_2 + c^3 + (b^2+b+1) c^2 + a-a') \pmod{w'^{-1} k[[w'^{-1}]]}.
\end{align}
Equating coefficients of $w'$ yields $h_{-1} = c^2$.
The $G$-equivariance of $\alpha$ implies 
\begin{align}
\nonumber
	\alpha(\sigma(z)) &= \sigma'(\alpha(z)) \\
\nonumber
	(z'+h) + (w'+f) + b &= (z' + w' + b') + \sigma'(h) \\
\nonumber
	h + f + b &= b' + \sigma'(h) \\
\label{E:equivariance congruence}
	h_{-1} w' + h_0 + c + b &\equiv b' + h_{-1}(w'+1) + h_0 \pmod{{w'}^{-1} k[[{w'}^{-1}]]} \\
\nonumber
	b - b' &= h_{-1} - c = c^2 - c = \wp(c).
\end{align}

$\impliedby$: Conversely, suppose that $b-b' = \wp(c)$ for some $c \in k$.
We must build a $G$-equivariant continuous isomorphism 
$\alpha \colon \widehat{K} \isomto \widehat{K}'$.
Choose $f \colonequals c + \sum_{i \ge 2} f_i {w'}^{-i}$ in $k[[v'^{-1}]]$ 
so that the value of $f_2$ makes the coefficient of ${w'}^0$
in~\eqref{E:h and f congruence}, namely the constant term, equal to $0$.
The coefficient of ${w'}^2$ in~\eqref{E:h and f congruence} is 
$c+ \wp(\wp(c)) = c^4$.
So \eqref{E:h and f congruence} simplifies to
\[
	h^2 - h \equiv c^4 {w'}^2 + c^2 {w'} \pmod{{w'}^{-1} k[[{w'}^{-1}]]}.
\]
Thus we may choose $h \colonequals c^2 w' + \sum_{i \ge 1} h_i {w'}^{-i}$
so that \eqref{E:h and f} holds.
Define $\alpha \colon k((w^{-1})) \to k(({w'}^{-1}))$
by $\alpha(w) \colonequals w'+f$.
Equation~\eqref{E:h and f} implies that $\alpha$
extends to $\alpha \colon \widehat{K} \to \widehat{K}'$
by setting $\alpha(z) \colonequals z'+h$.
Then $\alpha|_{k((w^{-1}))}$ is $G$-equivariant
since $(w'+1) + f = (w'+f) + 1$.
In other words, 
$\sigma^{-1} \alpha^{-1} \sigma' \alpha \in \Gal(\widehat{K}/k((w^{-1}))) = \{1,\sigma^2\}$.
If $\sigma^{-1} \alpha^{-1} \sigma' \alpha = \sigma^2$,
then
\begin{align*}
	\alpha \sigma^3 &= \sigma' \alpha \\
	\alpha(\sigma^3(z)) &= \sigma'(\alpha(z)) \\
	\alpha(z + w + b + 1) &= \sigma'(z'+h) \\
	 (z' + h) + (w' + f) + b + 1 &= (z'+w'+b') + \sigma'(h);
\end{align*}
by the calculation leading to~\eqref{E:equivariance congruence},
this is off by $1$ modulo ${w'}^{-1} k[[{w'}^{-1}]]$.
Thus $\sigma^{-1} \alpha^{-1} \sigma' \alpha = 1$ instead.
In other words, $\alpha$ is $G$-equivariant.
\end{proof}

\begin{remark}
\label{R:inverse}
Changing $b$ to $b+1$ does not change the curve $E_{a,b}$,
but it changes $\sigma$ to $\sigma^{-1}$.
Thus $\sigma$ and $\sigma^{-1}$ are conjugate in $\Aut(k[[t]])$
if and only if $1 \in \wp(k)$, i.e., 
if and only if $k$ contains a primitive cube root of unity.
\end{remark}

Combining Theorems \ref{T:order p^n}\eqref{I:cyclic of order p^n for n at least 2} and~\ref{T:surjection for order 4}
proves Theorem~\ref{T:almost rational} (and a little more).

\section{Constructions of Harbater--Katz--Gabber curves}
\label{S:constructions}

In this section we construct some examples needed for the proofs of
Theorems~\ref{T:solvable} and \ref{thm:kgsolv}.
Let $k$ be an algebraically closed field of characteristic $p > 0$.
Let $(Y,y)$ be an HKG $H$-curve over $k$.
If the $H$-action on $Y-\{y\}$ has a tamely ramified orbit,
let $S$ be that orbit; 
otherwise let $S$ be any $H$-orbit in $Y-\{y\}$.
Let $S'=S \union \{y\}$.
Let $m,n \in \Z_{\ge 1}$.
Suppose that $p \nmid n$, that $mn$ divides $|S'|$,
that the divisor $\sum_{s \in S'}(s-y)$ is principal,
and that for all $s \in S'$, the divisor $m(s-y)$ is principal.

Choose $f \in k(Y)^\times$ with divisor $\sum_{s \in S'} (s-y)$.
Let $\pi \colon X \to Y$ be the cover with $k(X)=k(Y)(z)$,
where $z$ satisfies $z^n=f$.
Let $C \colonequals \Aut(X/Y)$, so $C$ is cyclic of order~$n$.
Let $x$ be the point of $X(k)$ such that $\pi(x)=y$.
Let $G \colonequals \{\gamma \in \Aut(X) : \gamma|_{k(Y)} \in H\}$.

\begin{proposition}
\label{prop:constructautos}
Let $k,Y,H,S',n,X,C,G$ be as above.
\begin{enumerate}[\upshape (a)]
\item \label{I:automorphisms lift}
Every automorphism of $Y$ preserving $S'$
lifts to an automorphism of $X$ (in $n$ ways).
\item \label{I:CGH}
The sequence $1 \to C \to G \to H \to 1$ is exact.
\item \label{I:lift is HKG}
We have that $(X,x)$ is an HKG $G$-curve.
\end{enumerate}
\end{proposition}

\begin{proof}
\hfill
\begin{enumerate}[\upshape (a)]
\item 
Suppose that $\alpha \in \Aut(Y)$ preserves $S'$.
Then $\operatorname{div}({}^\alpha \! f/f) = (|S|+1)({}^\alpha y - y)$,
which is $n$ times an integer multiple of 
the principal divisor $m({}^\alpha y - y)$,
so ${}^\alpha \! f/f = g^n$ for some $g \in k(Y)^\times$.
Extend $\alpha$ to an automorphism of $k(X)$
by defining ${}^\alpha z \colonequals g z$;
this is well-defined since the relation $z^n=f$ is preserved.
Given one lift, all others are obtained by composing with elements of $C$.
\item 
Only the surjectivity of $G \to H$ is nontrivial,
and that follows from~\eqref{I:automorphisms lift}.
\item 
The quotient $X/G$ is isomorphic to $(X/C)/(G/C) = Y/H$,
which is of genus~$0$.
In the covers $X \to X/C \isom Y \to X/G \isom Y/H$,
all the ramification occurs above and below $S'$.
The valuation of $f$ at each point of $S'$
is $1 \bmod n$, so $X \to Y$ is totally ramified above $S'$.
Hence each ramified $G$-orbit in $X$ maps bijectively to an $H$-orbit in $Y$,
and each nontrivial inertia group in $G$ 
is an extension of a nontrivial inertia group of $H$ by $C$.
Thus, outside the totally ramified $G$-orbit $\{x\}$,
there is at most one ramified $G$-orbit and it is tamely ramified.\qedhere
\end{enumerate}
\end{proof}

\begin{example}
\label{Ex:building on genus 0}
Let $(Y,y)=(\PP^1,\infty)$, with coordinate function $t \in k(\PP^1)$.
Let $H \le \PGL_2(\F_q)$ be a group fixing $\infty$ 
and acting transitively on $\Aff^1(\F_q)$.
(One example is 
$H \colonequals \begin{pmatrix} 1 & \F_q \\ 0 & 1 \end{pmatrix}$.)
Let $n$ be a positive divisor of $q+1$.
Then the curve $z^n = t^q - t$ equipped with the point above $\infty$
is an HKG $G$-curve,
where $G$ is the set of automorphisms lifting those in $H$.
(Here $S'=\PP^1(\F_q)$, $m=1$, and $f = t^q-t \in k(\PP^1)$.
Degree~$0$ divisors on $\PP^1$ are automatically principal.)
\end{example}

\begin{example}
\label{Ex:building on genus 1, p=2}
Let $p=2$.
Let $(Y,y)$ be the $j$-invariant $0$ elliptic curve $u^2+u=t^3$ 
with its identity,
so $\#\Aut(Y,y)=24$ \cite[Chapter~3, \S6]{husemoeller}.
Let $H$ be $\Aut(Y,y)$ or its Sylow $2$-subgroup.
Then $k(Y)(\sqrt[3]{t^4+t})$ is the function field of an HKG $G$-curve $X$,
for an extension $G$ of $H$ by a cyclic group of order~$3$.
(Here $S'=Y(\F_4)$, 
which is also the set of $3$-torsion points on $Y$, 
and $m=n=3$, and $f = t^4+t$.)
Eliminating $t$ by cubing $z^3=t^4+t$ and substituting $t^3=u^2+u$
leads to the equation $z^9 = (u^2+u)(u^2+u+1)^3$ for $X$.
\end{example}

\begin{example}
\label{Ex:building on genus 1, p=3}
Let $p=3$.
Let $(Y,y)$ be the $j$-invariant $0$ elliptic curve $u^2=t^3-t$
with its identity,
so $\#\Aut(Y,y)=12$ \cite[Chapter~3, \S5]{husemoeller}.
Let $H$ be a group between $\Aut(Y,y)$ and its Sylow $3$-subgroup.
Then $k(Y)(\sqrt{u})$ is the function field of an HKG $G$-curve $X$,
for an extension $G$ of $H$ by a cyclic group of order~$2$.
(Here $S'$ is the set of $2$-torsion points on $Y$,
and $m=n=2$, and $f = u$.)
Thus $X$ has affine equation $z^4=t^3-t$.
(This curve is isomorphic to the curve in 
Example~\ref{Ex:building on genus 0} for $q=3$,
but $|C|$ here is $2$ instead of $4$.)
\end{example}

\section{Harbater--Katz--Gabber curves with extra automorphisms}
\label{S:extra automorphisms}

We return to assuming only that $k$ is perfect of characteristic $p$.
Throughout this section, $(X,x)$ is an HKG $G$-curve over $k$, 
and $J$ is a finite group such that $G \le J \le \Aut(X)$.
Let $J_x$ be the decomposition group of $x$ in $J$.
Choose Sylow $p$-subgroups $P \le P_x \le P_J$ 
of $G \le J_x \le J$, respectively.
In fact, $P \le G$ is uniquely determined
since $G$ is cyclic mod~$p$ by Lemma~\ref{L:G is cyclic mod p};
similarly $P_x \le J_x$ is uniquely determined.

\subsection{General results}

\begin{proof}[Proof of Theorem~\ref{T:J fixes x}]
If $(X,x)$ is an HKG $J$-curve, then $J$ fixes $x$, by definition.

Now suppose that $J$ fixes $x$.
By Lemma~\ref{L:G is cyclic mod p}, $J$ is cyclic mod~$p$.
By Proposition~\ref{P:G-curve vs. P-curve}\eqref{I:G-curve}$\Rightarrow$\eqref{I:P-curve},
$(X,x)$ is an HKG $P$-curve.
Identify $X/P$ with $\PP^1_k$ so that $x$ maps to $\infty \in X/P \isom \PP^1_k$.

\emph{Case 1: $J$ normalizes $G$.}
Then $J$ normalizes also the unique Sylow $p$-subgroup $P$ of $G$.
In particular, $P$ is normal in $P_J$.
If a $p$-group acts on $\PP^1_k$ fixing $\infty$,
it must act by translations on $\Aff^1_k$;
applying this to the action of $P_J/P$ on $X/P$ shows that
$X/P \to X/P_J$ is unramified outside $\infty$.
Also, $X \to X/P$ is unramified outside $x$.
Thus the composition $X \to X/P \to X/P_J$ is unramified outside $x$.
On the other hand, $X/P_J$ is dominated by $X/P$, so $g_{X/P_J}=0$.
By Proposition~\ref{P:G-curve vs. P-curve}\eqref{I:strong P-curve}$\Rightarrow$\eqref{I:G-curve},
$(X,x)$ is an HKG $J$-curve.

\emph{Case 2: $J$ is arbitrary.}
There exists a chain of subgroups beginning at $P$ and ending at $P_J$,
each normal in the next.
Ascending the chain, applying Case~1 at each step, 
shows that $(X,x)$ is an HKG curve for each
group in this chain, and in particular for $P_J$.
By Proposition~\ref{P:G-curve vs. P-curve}\eqref{I:P-curve}$\Rightarrow$\eqref{I:G-curve},
$(X,x)$ is also an HKG $J$-curve.
\end{proof}

\begin{corollary}
\label{C:J_x-curve and P_x-curve}
We have that $(X,x)$ is an HKG $J_x$-curve and an HKG $P_x$-curve.
\end{corollary}

\begin{proof}
Apply Theorem~\ref{T:J fixes x} with $J_x$ in place of $J$.
Then apply Proposition~\ref{P:G-curve vs. P-curve}\eqref{I:G-curve}$\Rightarrow$\eqref{I:P-curve}.
\end{proof}

\begin{lemma}
\label{L:maximal p' subgroup}
Among $p'$-subgroups of $J_x$ that are normal in $J$,
there is a unique maximal one; call it $C$.
Then $C$ is cyclic, and central in $J_x$.
\end{lemma}

\begin{proof}
Let $C$ be the group generated by all $p'$-subgroups of $J_x$ 
that are normal in $J$.
Then $C$ is another group of the same type, so it is the unique maximal one.
By Lemma~\ref{L:G is cyclic mod p}, $J_x$ is cyclic mod~$p$, 
so $J_x/P_x$ is cyclic.
Since $C$ is a $p'$-group, $C \to J_x/P_x$ is injective.
Thus $C$ is cyclic.
The injective homomorphism $C \to J_x/P_x$ 
respects the conjugation action of $J_x$ on each group.
Since $J_x/P_x$ is abelian, the action on $J_x/P_x$ is trivial.
Thus the action on $C$ is trivial too; i.e., $C$ is central in $J_x$.
\end{proof}

\subsection{Low genus cases}
\label{S:low genus}

Define $A \colonequals \Aut(X,x)$, so $G \le A$.
By Theorem~\ref{T:J fixes x}, $(X,x)$ is an HKG $J$-curve
if and only if $J \le A$.
When $g_X \le 1$, we can describe $A$ very explicitly.

\begin{example}
\label{Ex:genus 0}
Suppose that $g_X=0$.
Then $(X,x) \isom (\PP^1_k,\infty)$.
Thus $\Aut(X) \isom \PGL_2(k)$, and $A$ is identified with 
the image in $\PGL_2(k)$ of 
the group of upper triangular matrices in $\GL_2(k)$.
\end{example}

\begin{example}
\label{Ex:genus 1}
Suppose that $g_X=1$.
Then $(X,x)$ is an elliptic curve,
and $\Aut(X) \isom X(k) \rtimes A$.
Let $\calA \colonequals \Aut(X_{\kbar},x)$
be the automorphism group of the elliptic curve over $\kbar$.
Now $p$ divides $|G|$, since otherwise it follows from
Example~\ref{Ex:p'-curve} that $g_X=0$.
Thus $G$ contains an order~$p$ element, 
which by the HKG property has a unique fixed point.
Since $G \le A \le \calA$, the group $\calA$ also
contains such an element.
By the computation of $\calA$ 
(in~\cite[Chapter~3]{husemoeller}, for instance),
$p$ is $2$ or~$3$, and $X$ is supersingular,
so $X$ has $j$-invariant~$0$.
Explicitly:
\begin{itemize}
\item If $p=2$, 
then $\calA \isom \SL_2(\F_3) \isom Q_8 \rtimes \Z/3\Z$ (order $24$),
and $G$ is $\Z/2\Z$, $\Z/4\Z$, $Q_8$, or $\SL_2(\F_3)$.
\item If $p=3$,
then $\calA \isom \Z/3\Z \rtimes \Z/4\Z$ (order $12$),
and $G$ is $\Z/3\Z$, $\Z/6\Z$, or $\Z/3\Z \rtimes \Z/4\Z$.
\end{itemize}
Because of Corollary~\ref{C:J_x-curve and P_x-curve},
the statement about $G$ is valid also for $J_x$.
\end{example}

\subsection{Cases in which $p$ divides $|G|$}

If $p$ divides $|G|$, then we can strengthen Theorem~\ref{T:J fixes x}:
see Theorem~\ref{thm:cycp} and Corollary~\ref{cor:Jx} below.

\begin{lemma}
\label{L:G is normal in J}
If $p$ divides $|G|$ and $G$ is normal in $J$,
then $J$ fixes $x$.
\end{lemma}

\begin{proof}
Ramification outside $x$ is tame,
so if $p$ divides $|G|$, then $x$ is the unique point fixed by $G$.
If, in addition, $J$ normalizes $G$,
then $J$ must fix this point.
\end{proof}

\begin{theorem}
\label{thm:cycp}
If $p$ divides $|G|$, then the following are equivalent:
\begin{enumerate}[\upshape (i)]
\item \label{I:HKG J-curve}
$(X,x)$ is an HKG $J$-curve.
 \item \label{I:J fixes x}
 $J$ fixes $x$.
 \item \label{I:cyclic mod p}
 $J$ is cyclic mod~$p$.
 \end{enumerate}
\end{theorem}

\begin{proof}\hfill

\eqref{I:HKG J-curve}$\Leftrightarrow$\eqref{I:J fixes x}:
This is Theorem~\ref{T:J fixes x}.

\eqref{I:J fixes x}$\Rightarrow$\eqref{I:cyclic mod p}:
This is Lemma~\ref{L:G is cyclic mod p}.

\eqref{I:cyclic mod p}$\Rightarrow$\eqref{I:HKG J-curve}:
By Proposition~\ref{P:G-curve vs. P-curve}\eqref{I:G-curve}$\Rightarrow$\eqref{I:P-curve},
$(X,x)$ is an HKG $P$-curve.
Again choose a chain of subgroups beginning at $P$ and ending at $P_J$,
each normal in the next.
Since $J$ is cyclic mod~$p$, we may append $J$ to the end of this chain.
Applying Lemma~\ref{L:G is normal in J} and Theorem~\ref{T:J fixes x}
to each step of this chain shows that for each group $K$ in this chain,
$K$ fixes $x$ and $(X,x)$ is an HKG $K$-curve.
\end{proof}

\begin{corollary}
\label{cor:Jx}
If $p$ divides $|G|$, then
\begin{enumerate}[\upshape (a)]
\item\label{I:P_x=P_J}
$P_x=P_J$.
\item\label{I:p does not divide index}
The prime $p$ does not divide the index $(J:J_x)$.
\item\label{I:normalizer of P_x}
If $j \in J_x$, then ${}^j P_x = P_x$.
\item\label{I:TI}
If $j \notin J_x$, then ${}^j P_x \intersect P_x = 1$.
\item \label{I:normal p-subgroup}
If $J$ contains a nontrivial normal $p$-subgroup $A$,
then $(X,x)$ is an HKG $J$-curve.
\end{enumerate}
\end{corollary}

\begin{proof}
\hfill
\begin{enumerate}[\upshape (a)]
\item 
Since $p$ divides $|P_x|$ and $P_J$ is cyclic mod~$p$,
Corollary~\ref{C:J_x-curve and P_x-curve}
and Theorem~\ref{thm:cycp}\eqref{I:cyclic mod p}$\Rightarrow$\eqref{I:J fixes x} 
imply that $P_J$ fixes $x$.
Thus $P_J \le P_x$, so $P_x=P_J$.
\item The exponent of $p$ in each of $|J_x|$, $|P_x|$, $|P_J|$, $|J|$ 
is the same.
\item 
By Lemma~\ref{L:G is cyclic mod p}, $J_x$ is cyclic mod~$p$,
so $P_x$ is normal in $J_x$.
\item 
A nontrivial element of $P_x \intersect {}^j P_x$ 
would be an element of $p$-power order fixing both $x$ and $jx$, 
contradicting the definition of HKG $J_x$-curve.
\item 
The group $A$ is contained in every Sylow $p$-subgroup of $J$;
in particular, $A \le P_J = P_x$.
This contradicts~\eqref{I:TI} unless $J_x=J$.
By Theorem~\ref{thm:cycp}\eqref{I:J fixes x}$\Rightarrow$\eqref{I:HKG J-curve},
$(X,x)$ is an HKG $J$-curve. \qedhere
\end{enumerate}
\end{proof}

\begin{lemma}
\label{la:absub}
Suppose that $g_X>1$.
Let $A \le J$ be an elementary abelian $\ell$-subgroup for some prime $\ell$.
Suppose that $P_x$ normalizes $A$.
Then $A \le J_x$.
\end{lemma}

\begin{proof}
It follows from Example~\ref{Ex:p'-curve}  that $p$ divides $|G|$.
If $\ell=p$, then $P_x A$ is a $p$-subgroup of $J$,
but $P_x$ is a Sylow $p$-subgroup of $J$ 
by Corollary~\ref{cor:Jx}\eqref{I:P_x=P_J},
so $A \le P_x \le J_x$.

Now suppose that $\ell \ne p$.
The conjugation action of $P_x$ on $A$
leaves the group $A_x = J_x \intersect A$ invariant.
By Maschke's theorem, $A = A_x \times C$
for some other subgroup $C$ normalized by $P_x$.
Then $C_x=1$.
By Corollary~\ref{C:J_x-curve and P_x-curve}, 
$(X,x)$ is an HKG $P_x$-curve.
Since $P_x$ normalizes $C$, 
the quotient $X/C$ equipped with the image $y$ of $x$
and the induced $P_x$-action is another HKG $P_x$-curve.
Since $C_x=1$, we have $\mathfrak{d}_x(P_x)=\mathfrak{d}_y(P_x)$;
thus Proposition~\ref{P:G-curve vs. P-curve}\eqref{I:G-curve}$\Rightarrow$\eqref{I:genus of HKG curve}
implies that $g_X=g_{X/C}$.
Since $g_X>1$, this implies that $C=1$.
So $A=A_x \le J_x$.
\end{proof}

\subsection{Unmixed actions}

\begin{proof}[Proof of Theorem~\ref{thm:kgthm}]
By the base change property mentioned after Remark~\ref{R:katz},
we may assume that $k$ is algebraically closed.
By Corollary~\ref{C:J_x-curve and P_x-curve}, 
we may enlarge $G$ to assume that $G=J_x$.

First suppose that the action of $G$ 
has a nontrivially and tamely ramified orbit,
say $Gy$, where $y \in X(k)$.
The Hurwitz formula applied to $(X,G)$ gives
\begin{equation}
\label{E:Hurwitz equation for G}
	2g_X-2 = -2|G| + \mathfrak{d}_x(G) + |G/G_y|(|G_y|-1).
\end{equation}
Since the action of $J$ is unmixed, $Jx$ and $Jy$ are disjoint.
The Hurwitz formula for $(X,J)$ therefore gives 
\begin{equation}
\label{E:Hurwitz inequality for J}
	2g_X-2 \ge -2 |J| + |J/G| \mathfrak{d}_x(G) + |J/J_y|(|J_y|-1).
\end{equation}
Calculating $|J/G|$ times the equation~\eqref{E:Hurwitz equation for G} 
minus the inequality~\eqref{E:Hurwitz inequality for J} yields
\[
	(|J/G|-1)(2g_X-2) \le |J/J_y|-|J/G_y| \le 0,
\]
because $G_y \le J_y$. 
Since $g_X>1$, this forces $J=G$.

If a nontrivially and tamely ramified orbit does not exist,
we repeat the proof while omitting the terms involving $y$.
\end{proof}

\subsection{Mixed actions}

Here is an example, mentioned to us by Rachel Pries, that shows that
Theorem~\ref{thm:kgthm}  need not hold if the action of $J$ is mixed.  

\begin{example}
Let $n$ be a power of $p$; assume that $n>2$.
Let $k=\F_{n^6}$.
Let $\calX$ be the curve over $k$ constructed by
Giulietti and Korchm\'aros in~\cite{GiuKor};
it is denoted $\mathcal{C}_3$ in~\cite{Rachel}. 
Let $J=\Aut(\calX)$.
Let $G$ be a Sylow $p$-subgroup of $J$; by \cite[Theorem~7]{GiuKor}, $|G|=n^3$.
Then $\calX$ is an HKG $G$-curve 
by \cite[Lemma~2.5 and proof of Proposition~3.12]{Rachel},
and $g_{\calX}>1$ by~\cite[Thm. 2]{GiuKor}.
Taking $\sigma$ in Definition~\ref{D:mixitup}
to be the automorphism denoted $\tilde{W}$ on~\cite[p. 238]{GiuKor} 
shows that the action of $J$ on $\calX$ is mixed.
In fact, \cite[Theorem~7]{GiuKor} shows that $J$ fixes no $k$-point of $\calX$,
so the conclusion of Theorem~\ref{thm:kgthm} does not hold.
\end{example}

\subsection{Solvable groups}

Here we prove Theorem~\ref{T:solvable}. 
If $p$ does not divide $|G|$, then Example~\ref{Ex:p'-curve}
shows that $X \isom \PP^1_k$, 
so the conclusion of Theorem~\ref{T:solvable} holds. 
For the remainder of this section, we assume that $p$ divides $|G|$.
In this case we prove Theorem~\ref{T:solvable} 
in the stronger form of Theorem~\ref{thm:kgsolv},
which assumes a hypothesis weaker than solvability of $J$.
We retain the notation set at the beginning of 
Section~\ref{S:extra automorphisms},
and let $C$ denote the maximal $p'$-subgroup of $J_x$ that is normal in $J$, 
as in Lemma~\ref{L:maximal p' subgroup}.

\begin{lemma}
\label{L:C not 1}
Suppose that $g_X>1$ and that $(X,x)$ is not an HKG $J$-curve.
If $J$ contains a nontrivial normal abelian subgroup,
then $C \ne 1$.
\end{lemma}

\begin{proof}
The last hypothesis implies that $J$ contains
a nontrivial normal elementary abelian $\ell$-subgroup $A$ 
for some prime $\ell$.
By Corollary~\ref{cor:Jx}\eqref{I:normal p-subgroup}, $\ell \ne p$.
By Lemma~\ref{la:absub}, $A \le J_x$.
Thus $1 \ne A \le C$.
\end{proof}

\begin{theorem}
\label{thm:kgsolv}
Suppose that $p$ divides $|G|$ and $(X,x)$ is not an HKG $J$-curve.
\begin{enumerate}[\upshape (a)]
\item \label{I:solvable g=0}
Suppose that $g_X=0$, so $\Aut(X) \isom \Aut(\PP^1_k) \isom \PGL_2(k)$.
Then $J$ is conjugate in $\PGL_2(k)$ to precisely one of the following groups:
\begin{itemize}
\item
$\PSL_2(\F_q)$ or $\PGL_2(\F_q)$ for some finite subfield $\F_q \leq k$ 
(these groups are the same if $p=2$); 
note that $\PSL_2(\F_q)$ is simple when $q>3$.
\item If $p=2$ and $m$ is an odd integer at least $5$ 
such that a primitive $m$th root of unity $\zeta \in \kbar$ 
satisfies $\zeta + \zeta^{-1} \in k$,
the dihedral group of order $2m$ generated by
$\begin{pmatrix} \zeta & 0 \\ 0 & \zeta^{-1} \end{pmatrix}$ and 
$\begin{pmatrix}0 & 1 \\ 1 & 0 \end{pmatrix}$ 
if $\zeta \in k$,
and generated by 
$\begin{pmatrix} \zeta+ \zeta^{-1}+1 & 1 \\ 1 & 1 \end{pmatrix}$ and 
$\begin{pmatrix}0 & 1 \\ 1 & 0 \end{pmatrix}$
if $\zeta \notin k$.
(The case $m=3$ is listed already, as $\PSL_2(\F_2)$.)
\item
If $p=3$ and $\F_9 \leq k$, 
a particular copy of the alternating group $A_5$ in $\PSL_2(\F_9)$ 
(all such copies are conjugate in $\PGL_2(\mathbb F_9)$); 
the group $A_5$ is simple.
\end{itemize}
Suppose, in addition, that $J$ contains a nontrivial normal abelian subgroup;
then $p	\in \{2,3\}$ and $|P_J|=p$, and if $J$ is conjugate
to $\PSL_2(\F_q)$ or $\PGL_2(\F_q)$, then $q=p$.
\item \label{I:solvable g=1}
Suppose that $g_X=1$.
Then $p$ is $2$ or $3$, 
and the limited possibilities for $X$ and $J_x$ 
are described in Example~\ref{Ex:genus 1}. 
The group $J$ is a semidirect product of $J_x$ 
with a finite abelian subgroup $T \leq X(k)$. 
\item \label{I:solvable g>1}
Suppose that $g_X>1$. Let $C \leq J$ be as in Lemma~\ref{L:maximal p' subgroup}.
Let $Y=X/C$, let $y$ be the image of $x$ under $X \to Y$,
and let $U = \Stab_{J/C}(y)$.
If $J/C$ contains a nontrivial normal abelian subgroup 
(automatic if $J$ is solvable),
then one of the following holds:
\begin{enumerate}[\upshape i.]
\item
$p=3$, $g_X=3$, $g_Y=0$, $C \isom \Z/4\Z$, 
$P_x \isom \Z/3\Z$, $(J:J_x)=4$, 
and $(X,x)$ is isomorphic over $\kbar$ to 
the curve $z^4 = t^3 u - t u^3$ in $\PP^2$ 
equipped with ${(t:u:z)} = {(1:0:0)}$,
which is the curve in Example~\ref{Ex:building on genus 0} with $q=3$.
Moreover, 
\[
	\PSL_2(\F_3) \le J/C \le \PGL_2(\F_3).
\]
\item
$p=2$, $g_X=10$, $g_Y=1$, $C \isom \Z/3\Z$, $P_x \isom Q_8$, 
$(J:J_x)=9$, 
and $(X,x)$ is isomorphic over $\kbar$ to the curve in 
Example~\ref{Ex:building on genus 1, p=2}.  
 The homomorphism $J \to J/C$ sends the subgroups $J_x \supset P_x$ to 
subgroups $J_x/C \supset P_x C/C$ of $U$.
Also, $P_x C/ C \isom P_x \isom Q_8$ and $U \isom \SL_2(\Z/3\Z)$,
and $U$ acts faithfully on the $3$-torsion subgroup $Y[3] \isom (\Z/3\Z)^2$
of the elliptic curve $(Y,y)$.  
The group $J/C$ satisfies
\[
	Y[3] \rtimes Q_8 \isom (\Z/3\Z)^2 \rtimes Q_8 \; \le \; J/C 
	\; \le \; (\Z/3\Z)^2 \rtimes \SL_2(\Z/3\Z) \isom Y[3] \rtimes U.
\]
\item
$p=3$, $g_X=3$, $g_Y=1$, $C \isom \Z/2\Z$, $P_x \isom \Z/3\Z$, $(J:J_x)=4$, 
and $(X,x)$ is isomorphic over $\kbar$ to 
the curve $z^4 = t^3 u - t u^3$ in $\PP^2$ 
equipped with ${(t:u:z)} = {(1:0:0)}$
as in Example~\ref{Ex:building on genus 1, p=3}.
The homomorphism $J \to J/C$
sends the subgroups $J_x \supset P_x$ to 
subgroups $J_x/C \supset P_x C/C$ of $U$.
Also $P_x C/C \isom P_x \isom \Z/3\Z$ and
$U \isom \Z/3\Z \rtimes \Z/4\Z$, 
and $U/Z(U)$ acts faithfully on the group $Y[2] \isom (\Z/2\Z)^2$.
The group $J/C$ satisfies
\[
Y[2]\rtimes \Z/3\Z = (\Z/2\Z)^2 \rtimes \Z / 3\Z \; \le \; J/C 
	\; \le \; (\Z/2\Z)^2 \rtimes (\Z/3\Z \rtimes \Z/4\Z)  = Y[2] \rtimes U.
\]
\end{enumerate}
In each of i.,\ ii.,\ and iii.,\ 
if $(X,x)$ is the curve over $\kbar$ specified, 
from Examples~\ref{Ex:building on genus 0}--\ref{Ex:building on genus 1, p=3},
then any group satisfying the displayed upper and lower bounds for $J/C$
is actually realized as $J/C$ for some subgroup $J \le \Aut(X)$
satisfying all the hypotheses.
\end{enumerate}
\end{theorem}

\begin{proof} \hfill

(a)
The groups listed in the statement of~\eqref{I:solvable g=0} 
are pairwise non-isomorphic, hence not conjugate.
Thus it remains to prove that $J$ is conjugate to one of them.
By Corollary~\ref{cor:Jx}\eqref{I:normal p-subgroup},
$J$ has no normal Sylow $p$-subgroup.
We will show that \emph{every} finite subgroup $J \le \PGL_2(k)$
with no normal Sylow $p$-subgroup 
is conjugate to a group listed in~\eqref{I:solvable g=0}.
This would follow immediately from \cite[Theorem~B]{faber},
but \cite{faber} has not yet been published,
so we now give a proof not relying on it.
We will use the exact sequence
\[
   1 \to \PSL_2(k) \to \PGL_2(k) \stackrel{\det}\to k^\times/k^{\times 2} \to 1.
\]

\emph{Case 1: $k$ is finite and $J \le \PSL_2(k)$.}
For finite $k$, the subgroups of $\PSL_2(k)$ up to conjugacy 
were calculated by Dickson \cite[\S260]{dickson}; 
see also \cite[Ch.2 \S8]{huppert}, \cite[Ch.3 \S6]{suzuki}. 
The ones with no normal Sylow $p$-subgroup 
are among those listed in~\eqref{I:solvable g=0}.
(Dickson sometimes lists \emph{two} $\PSL_2(k)$-conjugacy classes 
of subgroups of certain types, but his proof shows that
they map to a single $\PGL_2(k)$-conjugacy class.)

\emph{Case 2: $k$ is infinite and $J \le \PSL_2(k)$.}
Let $\widetilde{J}$ be the inverse image of $J$
under the finite extension $\SL_2(k) \surjects \PSL_2(k)$.
So $\widetilde{J}$ is finite.
The representation of $\widetilde{J}$ on $k^2$ is absolutely irreducible,
since otherwise $\widetilde{J}$ would inject into the group
$\begin{pmatrix} * & * \\ 0 & * \end{pmatrix}$
of $2 \times 2$ upper triangular invertible matrices over $\kbar$,
and $\widetilde{J}$ would have a normal Sylow $p$-subgroup
$\widetilde{J} \intersect 
\begin{pmatrix} 1 & * \\ 0 & 1 \end{pmatrix}$,
and $J$ would have one too, contrary to assumption. 
By \cite[Theorem 19.3]{feit}, 
this representation is definable over the field $k_0$ 
generated by the traces of the elements of $\widetilde{J}$. 
Each trace is a sum of roots of unity, so $k_0$ is finite.
Thus $J$ is conjugate in $\PGL_2(k)$ to a subgroup $J_0 \le \PGL_2(k_0)$. 
Conjugation does not change the determinant, 
so $J_0 \le \PSL_2(k_0)$.
By Case~1, $J_0$ is conjugate to a group in our list, so $J$ is too.

\emph{Case 3: $k$ is finite or infinite, and $J \le \PGL_2(k)$, but $J \nleq \PSL_2(k)$.}
If $p=2$, then, since $k$ is perfect, $k^\times = k^{\times 2}$,
so $\PGL_2(k)=\PSL_2(k)$.
Thus $p>2$.
Let $J' \colonequals J \intersect \PSL_2(k)$.
Then $J/J'$ injects into $k^\times/k^{\times 2}$, so $p \nmid (J:J')$.
The Sylow $p$-subgroups of $J'$ are the same as those of $J$,
so $J'$ has exactly one if and only if $J$ has exactly one;
i.e., $J'$ has a normal Sylow $p$-subgroup if and only if $J$ has one.
Since $J$ does not have one, neither does $J'$.
By Case~1, we may assume that $J'$ appears in our list.

The group $J$ is contained in the normalizer $N_{\PGL_2(k)}(J')$.
We now break into cases according to $J'$.
If $J'$ is $\PSL_2(\F_q)$ or $\PGL_2(\F_q)$ 
for some subfield $\F_q \le k$,
then $N_{\PGL_2(k)}(J') = \PGL_2(\F_q)$ 
by~\cite[\S255]{dickson} (the proof there works even if $k$ is infinite),
so $J=\PGL_2(\F_q)$, which is in our list.
Recall that $p>2$, so $J'$ is not dihedral.
Thus the only remaining possibility is that 
$J' \isom A_5 \le \PSL_2(\F_9) \le \PGL_2(k)$.
Let $\{ 1 , a \}$ be a subgroup of order $2$ 
in the image of $J$ in $k^\times/k^{\times 2}$ 
and let $J''$ be its inverse image in $J$. 
Then $J'' < \PSL_2(k(\sqrt{a}))$, 
so $J''$ should appear in our list, 
but $|J''|=120$ and there is no group of order $120$ there for $p=3$.

(b)
In the notation of Example~\ref{Ex:genus 1},
let $\psi \colon J \to A$ be the projection.
Let $T \colonequals \ker \psi \le X(k)$.
Since $X$ is supersingular, $T$ is a $p'$-group.
Let $\Jbar \colonequals \psi(J) \le A$.
Since $G \le \Jbar \le \calA$, 
the group $\Jbar$ is in the list of possibilities in Example~\ref{Ex:genus 1}
for $G$ given $p$.
Checking each case shows that its Sylow $p$-subgroup 
$\Pbar_J \colonequals \psi(P_J)$ is normal in $\Jbar$.
The action of $\Aut(X)$ on $X(k)$ restricts to the conjugation action
of $J$ on the abelian group $T$, which factors through $\Jbar$,
so $H^0(\Pbar_J,T) = T^{\Pbar_J} = T^{P_J} = 0$,
since $P_J$ has a unique fixed point on $X$.
Also, $H^i(\Pbar_J,T)=0$ for all $i \ge 1$,
since $|\Pbar_J|$ and $|T|$ are coprime.
Thus, by the Lyndon--Hochschild--Serre spectral sequence
applied to $\Pbar_J \lhd \Jbar$, 
we have $H^i(\Jbar,T)=0$ for all $i \ge 1$. 
Therefore the short exact sequence 
$0 \to T \to J \to \Jbar \to 1$ is split, and all splittings are conjugate.  
Let $K$ be the image of a splitting $\Jbar \to J$. 
Then $K$ contains a Sylow $p$-subgroup of $J$.  
Equivalently, some conjugate $K'$ of $K$ contains $P_J$. 
 Since $K' \isom \Jbar$ and $\Pbar_J$ is normal in $\Jbar$, 
the group $P_J$ is normal in $K'$.  
Since $x$ is the unique fixed point of $P_J$, 
this implies that $K'$ fixes $x$; i.e., $K' \le J_x$. 
On the other hand, $|K'| = |\Jbar| \ge |J_x|$ since $J_x \intersect T = \{e\}$.
Hence $K' = J_x$ and  $J =  T \rtimes J_x$.

(c) 
We may assume that $k$ is algebraically closed.
By Theorem~\ref{T:J fixes x}, $(X,x)$ is an HKG $J_x$-curve.
Then $(Y,y)$ is an HKG $J_x/C$-curve, but not an HKG $J/C$-curve 
since $J/C$ does not fix $y$.
If $g_Y>1$, then Lemma~\ref{L:C not 1} applied to $Y$
yields a nontrivial $p'$-subgroup $C_1 \le J_x/C$ that is normal in $J/C$,
and the inverse image of $C_1$ in $J$ is a $p'$-subgroup $C_2 \le J_x$ 
normal in $J$ with $C_2 \gneq C$,
contradicting the maximality of $C$.
Thus $g_Y \le 1$.
Since $g_X>1$, we have $C \ne 1$.
Let $n=|C|$.
Let $\zeta$ be a primitive $n$th root of unity in $k$.
Let $c$ be a generator of $C$.

By Lemma~\ref{L:maximal p' subgroup}, $C$ is central in $J_x$,
so $P_x C$ is a direct product.
By Corollary~\ref{C:J_x-curve and P_x-curve}, $X$ is an HKG $P_x$-curve.
Thus $X/P_x \isom \PP^1$, 
and the $P_x$-action on $X$ is totally ramified at $x$
and unramified elsewhere.
The action of $C$ on $X/P_x$ fixes the image of $x$,
so by Example~\ref{Ex:p'-curve},
the curves in the covering $X/P_x \to X/P_x C$
have function fields $k(z) \supseteq k(f)$, where $z^n=f$ 
and ${}^c \! z=\zeta z$. 
Powers of $z$ form a $k(X/P_xC)$-basis of eigenvectors 
for the action of $c$ on $k(X/P_x)$.

We may assume that the (totally ramified) image of $x$ in $X/P_x$ 
is the point $z=\infty$. We obtain a diagram of curves
\[
\xymatrix{
& X \ar[ld]^-C \ar[rd]_-{P_x}^-{\lefteqn{\text{\scriptsize totally ramified above $z=\infty$, unramified elsewhere}}} \\
Y \isom X/C \ar[rd]^-{P_x} && X/P_x \isom \PP^1_z \ar[ld]_-C^-{\lefteqn{\text{\scriptsize totally ramified above $f=\infty$, $f=0$}}} \\
& X/P_x C \isom \PP^1_f
}
\]
where the subscript on each $\PP^1$ indicates the generator of 
its function field, and the group labeling each morphism 
is the Galois group. 
The field $k(X)$ is the compositum of its subfields $k(Y)$ and $k(X/P_x)$. 

Let $S$ be the preimage of the point $f=0$ under $Y \to X/P_x C$,
and let $S' \colonequals S \union \{y\}$.
Comparing the $p$-power and prime-to-$p$ ramification on both sides
of the diagram shows that the point $f=\infty$ totally ramifies
in $X \to Y \to X/P_x C$,
while the point $f=0$ splits completely 
into a set $S$ of $|P_x|$ points of $Y$, each of which is totally ramified
in $X \to Y$.
Thus the extension $k(X) \supseteq k(Y)$ is Kummer and generated by 
the same $z$ as above, 
and powers of $z$ form a $k(Y)$-basis of eigenvectors 
for the action of $c$ on $k(X)$.
This extension is totally ramified above $S'$ and unramified elsewhere.
The divisor of $f$ on $Y$ is 
$S - |S| y = S' - |S'| y$,
where $S$ here denotes the divisor $\sum_{s \in S} s$,
and so on.

Let $j \in J$.
Since $C \lhd J$, the element $j$ acts on $Y$ and preserves 
the branch locus $S'$ of $X \to Y$.
Since $X \to Y$ is totally ramified above $S'$,
the automorphism $j$ fixes $x$ if and only if it fixes $y$.
Since $P_x$ acts transitively on $S$, and $J$ does not fix $x$ or $y$,
the set $S'$ is the $J$-orbit of $y$.
Thus 
\[
	(J:J_x) = |Jx| = |Jy| = |S'| = |P_x|+1.
\]
Suppose that $j \in J-J_x$, so ${}^j y \ne y$.
Then the divisor of ${}^j \! f/f$ on $Y$ is 
\[
	\left(S' - |S'| \, {}^j y \right)
	- \left(S' - |S'| y \right)
	= |S'|(y-{}^j y),
\]
which is nonzero. 
Since $C$ is cyclic and normal, 
$j^{-1}cj = c^r$ for some $r$, 
and hence ${}^c  ({}^j \! z/z)= {}^{jc^r} \!\! z / {}^c z = \zeta^{r-1} \left( {}^j \! z/z \right)$. 
Thus ${}^j \! z/z$ is a $\zeta^{r-1}$-eigenvector, 
so ${}^j \! z/z = z^{r-1} g$ for some $g \in k(Y)^\times$.
Taking $n$th powers yields ${}^j \! f/f = f^{r-1} g^n$.
The corresponding equation on divisors is 
\begin{equation}
\label{E:divisors}
	|S'|(y-{}^j \! y)= (r-1)(S' - |S'|y) +n \operatorname{div}(g).
\end{equation}
Considering the coefficient of a point of $S' - \{y,{}^j y\}$
shows that $r-1 \equiv 0 \pmod{n}$.
Then, considering the coefficient of $y$ shows that $n$ divides $|S'|$,
and dividing equation~\eqref{E:divisors} through by $n$ 
shows that $(|S'|/n)(y-{}^j y)$ is $\operatorname{div}(f^{(r-1)/n} g)$, 
a principal divisor.
If, moreover, $g_{Y}>0$, then 
a difference of points on $Y$ cannot be a principal divisor,
so $n \ne |S'|$.

\emph{Case 1: $g_{Y}=0$.}
Applying~\eqref{I:solvable g=0} to $Y$ shows that 
$p \in \{2,3\}$ and any Sylow $p$-subgroup of $J/C$ has order~$p$.
Since $C$ is a $p'$-group, $|P_J|=p$ too.
By Corollary~\ref{cor:Jx}\eqref{I:P_x=P_J}, $P_x=P_J$, 
so $|P_x|=p$, and $n$ divides $|S'|=p+1$.
Thus $(p,n)$ is $(2,3)$, $(3,2)$, or $(3,4)$.
The Hurwitz formula for $X \to Y$ yields
\[
	2g_X-2 = n(2 \cdot 0 - 2) + \sum_{s \in S'} (n-1) = -2n + (p+1)(n-1).
\]
Only the case $(p,n) = (3,4)$ yields $g_X>1$.
By~\eqref{I:solvable g=0},
we may choose an isomorphism $Y \isom \PP^1_t$
mapping $y$ to $\infty$
such that the $J/C$-action on $Y$ 
becomes the standard action of $\PSL_2(\F_3)$ or $\PGL_2(\F_3)$ on $\PP^1_t$.
Then $S' = Jy = \PP^1(\F_3)$.
Then $f$ has divisor $S' - 4y = \Aff^1(\F_3) - 3 \cdot \infty$ on $\PP^1$,
so $f = t^3-t$ up to an irrelevant scalar.
Since $k(X)=k(Y)(\sqrt[n]{f})$,
the curve $X$ has affine equation $z^4 = t^3-t$.
This is the same as the $q=3$ case of 
Example~\ref{Ex:building on genus 0}.

\emph{Case 2: $g_{Y}=1$.}
Applying~\eqref{I:solvable g=1} (i.e., Example~\ref{Ex:genus 1}) to $Y$
shows that either $p$ is $2$ and $|P_x|$ divides $8$, 
or $p=3$ and $|P_x|=3$; also, $Y$ has $j$-invariant~$0$.
Also, $n$ divides $|S'| =|P_x|+1$, but $n$ is not $1$ or $|S'|$.
Thus $(p,n,|P_x|,|S'|)$ is $(2,3,8,9)$ or $(3,2,3,4)$.
The Hurwitz formula as before gives $g_X=10$ or $g_X=3$, respectively.
Let $m=|S'|/n$.
Since $m(y-{}^j y)$ is principal for all $j \in J$,
if $y$ is chosen as the identity of the elliptic curve,
then the $J$-orbit $S'$ of $y$ is contained in the 
group $Y[m]$ of $m$-torsion points.
But in both cases, these sets have the same size $|S'| = m^2$.
Thus $S' = Y[m]$.

If $p=2$, the $j$-invariant $0$ curve $Y$ has equation $u^2+u=t^3$,
and $Y[3]-\{y\}$ is the set of points with $t \in \F_4$,
so $f=t^4+t$ up to an irrelevant scalar,
and $k(X)=k(Y)(\sqrt[3]{t^4+t})$.
Thus $X$ is the curve of Example~\ref{Ex:building on genus 1, p=2}.

If $p=3$, the $j$-invariant $0$ curve $Y$ has equation $u^2=t^3-t$,
and $Y[2]-\{y\}$ is the set of points with $u=0$,
so $f=u$ up to an irrelevant scalar,
and $k(X) = k(Y)(\sqrt{u}) = k(t)(\sqrt[4]{t^3-t})$.
Thus $X$ is the curve of Example~\ref{Ex:building on genus 1, p=3}.

Finally, Proposition~\ref{prop:constructautos} implies
that in each of i.,\ ii.,\ and iii.,\ 
any group satisfying the displayed upper and lower bounds,
viewed as a subgroup of $\Aut(Y)$, 
can be lifted to a suitable group $J$ of $\Aut(X)$.
\qedhere
\end{proof}

\begin{remark}
Suppose that $(X,x)$ is not an HKG $J$-curve, 
$g_X > 1$, 
and $P_J$ is not cyclic or generalized quaternion.
Then \cite[Theorem~3.16]{Rachel} shows that $J/C$ is almost simple 
with socle from a certain list of finite simple groups.
\end{remark}

\section*{Acknowledgement}

We thank the referee for suggestions on the exposition.

\end{document}